\documentclass{amsart}

\usepackage[english]{babel}
\usepackage{amsmath,amssymb,mathtools,enumitem,mathrsfs,caption,arydshln}

\usepackage{amsthm}
\usepackage{cite}
\usepackage{amsrefs}

\usepackage{varioref}
\usepackage{hyperref}
\usepackage[capitalise]{cleveref}

\newcommand{\mc}[1]{\mathcal{#1}}

\newcommand{\ms}[1]{\mathscr{#1}}
\newcommand{\ol}[1]{\overline{#1}}

\newcommand{\ti}[1]{\textit{#1}}
\newcommand{\tx}[1]{\textrm{#1}}
\newcommand{\bs}[1]{\boldsymbol{#1}}

\renewcommand{\Im}{\operatorname{Im}}

\newcommand{\C}{\mathbb{C}}
\newcommand{\D}{\mathbb{D}}

\newcommand{\R}{\mathbb{R}}

\newcommand{\dd}{d}

\newcommand{\ddt}{\frac{d}{dt}}

\newcommand{\del}{\delta}

\newcommand{\dk}{\,dk}

\newcommand{\dmu}{\,d\mu}

\newcommand{\ds}{\,ds}

\newcommand{\dx}{\,dx}

\newcommand{\eps}{\epsilon}

\newcommand{\Arg}{\operatorname{arg}}

\newcommand{\bnorm}[1]{\big\lVert#1\big\rVert}
\newcommand{\con}{\ms{C}_{\bs{e}}}
\newcommand{\F}{\mc{F}}
\newcommand{\M}{\mc{M}}
\newcommand{\norm}[1]{\left\lVert#1\right\rVert}
\newcommand{\schwar}{\mc{S}}
\newcommand{\sm}{\smallsetminus}
\newcommand{\snorm}[1]{\lVert#1\rVert}
\newcommand{\wh}[1]{\widehat{#1}}

\newcommand{\wt}[1]{\widetilde{#1}}

\theoremstyle{plain}\newtheorem{thm}{Theorem}[section]
\theoremstyle{plain}\newtheorem{cor}[thm]{Corollary}
\theoremstyle{plain}\newtheorem{prop}[thm]{Proposition}
\theoremstyle{plain}\newtheorem{lem}[thm]{Lemma}
\theoremstyle{definition}\newtheorem{dfn}[thm]{Definition}
\theoremstyle{definition}\newtheorem{prob}[thm]{Problem}
\theoremstyle{remark}
\theoremstyle{definition}\newtheorem*{ack}{Acknowledgments}
\theoremstyle{definition}

\crefname{sec}{Section}{Sections}
\crefname{dfn}{Definition}{Definitions}
\crefname{lem}{Lemma}{Lemmas}
\crefname{prop}{Proposition}{Propositions}
\crefname{thm}{Theorem}{Theorems}
\crefname{cor}{Corollary}{Corollaries}
\crefname{fig}{Figure}{Figures}

\numberwithin{equation}{section}

\allowdisplaybreaks

\begin{document}

\title[Constrained minimizers of {K}d{V}]{Multisolitons are the unique constrained minimizers of the {K}d{V} conserved quantities}

\author{Thierry Laurens}
\address{Thierry Laurens \\
Department of Mathematics\\
University of California, Los Angeles, CA 90095, USA}
\email{laurenst@math.ucla.edu}

\begin{abstract}
We consider the following variational problem: minimize the $(n+1)$st polynomial conserved quantity of KdV over $H^n(\R)$ with the first $n$ conserved quantities constrained.  Maddocks and Sachs~\cite{Maddocks1993} used that $n$-solitons are \emph{local} minimizers for this problem in order to prove that $n$-solitons are orbitally stable in $H^n(\R)$.

Given $n$ constraints that are attainable by an $n$-soliton, we show that there is a unique set of $n$ amplitude parameters so that the corresponding multisolitons satisfy the constraints.  Moreover, we prove that these multisolitons are the \emph{unique global} constrained minimizers.  We then use this variational characterization to provide a new proof of the orbital stability result from~\cite{Maddocks1993} via concentration compactness.

In the case when the constraints can be attained by functions in $H^n(\R)$ but not by an $n$-soliton, we discover new behavior for minimizing sequences.
\end{abstract}

\maketitle



\section{Introduction}

The Korteweg--de Vries (KdV) equation
\begin{equation}
\ddt u = - u''' + 6uu'
\label{eq:kdv}
\end{equation}
(where $u' = \partial_xu$ denotes spatial differentiation) was derived over a century ago as a model for surface waves in a shallow channel of water.  Although the equation was first proposed by Boussinesq~\cite{Boussinesq1872}, it did not gain traction until Korteweg and de Vries~\cite{Korteweg1895} used the explicit solutions
\begin{equation}
q(t,x) = - 2\beta^2 \operatorname{sech}^2 [ \beta ( x - 4\beta^2 t- x_0 ) ]
\label{eq:soliton}
\end{equation}
for $\beta>0$ and $x_0\in\R$ to explain the empirical observation of solitary traveling waves.

The solutions~\eqref{eq:soliton} are now commonly referred to as \emph{solitons}, due to their particle-like behavior during interactions.  This name was coined by Kruskal and Zabusky~\cite{Zabusky1965} when they numerically observed that two colliding solitons emerge with unchanged profiles and speeds.  The interaction is nevertheless nonlinear, and this is manifested in a spatial shift of both waves in comparison to their initial trajectories.

We now know that this particle-like interaction can be well-approximated by a single explicit solution which resembles two solitons with distinct amplitudes as $t\to\pm\infty$.  In fact, together with the single soliton solutions~\eqref{eq:soliton}, these are the beginning of a family of solutions called \emph{multisolitons} which describe the interaction of an arbitrary number of distinct soliton profiles:
\begin{dfn}[Multisoliton solutions]
Fix $N\geq 1$.  Given $\beta_1,\dots,\beta_N>0$ distinct and $c_1,\dots,c_N\in\R$, the \emph{multisoliton of degree $N$} (or \emph{$N$-soliton}) with these parameters is
\begin{equation}
Q_{\bs{\beta},\bs{c}} (x) = - 2 \tfrac{d^2}{d x^2} \log\det [A(x)] ,
\label{eq:multisoliton}
\end{equation}
where $A(x)$ is the $N\times N$ matrix with entries
\begin{equation*}
A_{jk} (x) = \del_{jk} + \tfrac{1}{\beta_j+\beta_k} e^{- \beta_j(x-c_j) - \beta_k(x-c_k)} .
\end{equation*}
The unique solution to KdV~\eqref{eq:kdv} with initial data $u(0,x) = Q_{\bs{\beta},\bs{c}}(x)$ is
\begin{equation*}
u(t,x) = Q_{\bs{\beta},\bs{c}(t)}(x) , \quad\tx{where}\quad c_j(t) = c_j + 4\beta_j^2 t .
\end{equation*}
We define the \emph{multisoliton of degree zero} to be the zero function.
\end{dfn}

Formula~\eqref{eq:multisoliton} was first discovered in~\cite{Kay1956} as part of a study of potentials for one-dimensional Schr\"odinger operators with vanishing reflection coefficient.  Multisolitons have since been thoroughly examined by means of inverse scattering theory; see for example~\cites{Gardner1967,Gardner1974,Hirota1971,Tanaka1972/73,Wadati1972,Zakharov1971}.

The variational study of solitons~\eqref{eq:soliton} dates back to Boussinesq~\cite{Boussinesq1877}.  In addition to noting the conservation of the momentum and energy functionals
\begin{equation}
E_1(u) = \int_{-\infty}^\infty \tfrac{1}{2} u^2 \dx
\quad\tx{and}\quad 
E_2(u) = \int_{-\infty}^\infty \big[ \tfrac{1}{2} (u')^2 + u^3 \big] \dx ,
\label{eq:e1 and e2}
\end{equation}
he also realized that solitons are critical points for $E_2$ with $E_1$ constrained.  We now know that the functionals~\eqref{eq:e1 and e2} are the beginning of an infinite sequence of conserved quantities~\cite{Miura1968}, because KdV is a completely integrable system.  Their densities are defined recursively by~\cite{Zaharov1971}:
\begin{equation*}
\sigma_1(x) = u(x) , \qquad \sigma_{m+1}(x) = -\sigma'_m(x) - \sum_{j=1}^{m-1}\sigma_j(x)\sigma_{m-j}(x) .
\end{equation*}
For even $m$ the density $\sigma_m$ is a complete derivative, but for odd $m$ we obtain a nontrivial conserved quantity
\begin{equation}
E_n(u) = (-1)^{n} \tfrac{1}{2} \int_{-\infty}^\infty \sigma_{2n+1}(x)\dx
\label{eq:en 3}
\end{equation}
whose density is a polynomial in $u,u',\dots,u^{(n-1)}$.  The first two functionals of this sequence are~\eqref{eq:e1 and e2}, and the next one is given by
\begin{equation}
E_3(u) = \int_{-\infty}^\infty \big[ \tfrac{1}{2} (u'')^2 + 5u(u')^2 + \tfrac{5}{2} u^4 \big] \dx .
\label{eq:e3}
\end{equation}

Nearly a century after Boussinesq's work, Lax~\cite{Lax1968} studied two-soliton solutions of KdV in an effort to explain the particle-like interaction of solitons.  Although he does not explicitly state it (until his later work~\cite{Lax1975}), his ODE for the two-soliton is the Euler--Lagrange equation for a critical point of $E_3$ with $E_1$ and $E_2$ constrained.  In general, the $n$-soliton is known to be a critical point for the following variational problem:
\begin{prob}
Given an integer $n\geq 0$ and constraints $e_1,\dots,e_n$, minimize $E_{n+1}(u)$ over the set
\begin{equation*}
\con = \{ u \in H^n(\R) : E_1(u) = e_1 ,\ \dots,\ E_n(u) = e_n \} .
\end{equation*}
\end{prob}

In the case $n=0$ there are no constraints, and the minimizer of the $L^2$-norm over the space $L^2(\R)$ is simply the zero-soliton $q(x) \equiv 0$.  We include this trivial observation because it will provide a convenient base case for an induction argument.

This variational problem has been heavily studied within the context of multisoliton stability.  The first result was produced by Benjamin~\cite{Benjamin1972}, who proved that solitons are orbitally stable in $H^1(\R)$: solutions that start close to a soliton profile in $H^1(\R)$ remain close to a soliton profile for all time.  This was the introduction of a widely-applicable variational argument (cf.~\cite{Weinstein1986}) based on the fact that solitons are local constrained minimizers of $E_2$.  Benjamin's original work did contain some mathematical gaps however, but these were later resolved by Bona~\cite{Bona1975}.

Solitons are not merely local minimizers for this problem, but are \emph{global} minimizers~\cite{Albert1999} (see also \cref{thm:var char} below).  In order to employ this to prove orbital stability though, we need to know that profiles that almost minimize $E_2$ are close to a minimizing soliton.   In general, we cannot expect minimizing sequences to admit convergent subsequences, because the manifold of minimizing solitons is translation-invariant and hence non-compact.  This issue was solved by Cazenave and Lions~\cite{Cazenave1982} for a variety of NLS-like equations by a concentration compactness principle:~minimizing sequences are precompact \emph{modulo translations}.  This powerful method is now the dominant way of proving orbital stability, but it has not yet been successfully applied to this variational problem
because it requires a global understanding.

Nevertheless, Maddocks and Sachs~\cite{Maddocks1993} discovered that $n$-solitons are orbitally stable in $H^n(\R)$.  They first showed that $n$-solitons are indeed local minimizers of $E_{n+1}$ with $E_1,\dots,E_n$ constrained.  Then their argument relied on a careful study of the Hessian of $E_{n+1}$ on the manifold of minimizing $n$-solitons in directions tangent and perpendicular to the constraints.  A key ingredient in analyzing the tangent directions was the formula
\begin{equation}
E_n(Q_{\bs{\beta},\bs{c}}) = (-1)^{n+1} \tfrac{2^{2n+1}}{2n+1} \sum_{m=1}^N \beta_m^{2n+1}
\label{eq:en multisoliton}
\end{equation}
for the value of $E_n$ at an $N$-soliton, which is independent of $\bs{c}$.  This local analysis then implied the global result by employing the commuting flows of $E_1,\dots,E_n$.

However, the variational problem remains unsolved:~are multisolitons \emph{global} constrained minimizers of $E_{n+1}$?  If so, are they unique?  In particular, affirmative answers to this problem would be a significant step towards applying concentration compactness to prove the orbital stability of multisolitons.

More generally, we would like to understand all solutions to this natural problem because they are basic building blocks.  Indeed, special solutions to completely integrable models elucidate an avenue to a low-complexity understanding of the dynamics.  A critical point for this variational problem must satisfy the Euler--Lagrange equation
\begin{equation}
\nabla E_{n+1}(u) = \lambda_1 \nabla E_1(u) + \lambda_2 \nabla E_2(u) + \dots + \lambda_n \nabla E_n(u) .
\label{eq:novikov}
\end{equation}
Solutions to these equations are called \emph{algebro-geometric solutions}, and they are fundamental objects for integrable systems~\cites{Gesztesy2003,Gesztesy2008}; in fact, solitons were discovered via~\eqref{eq:novikov}.  Naturally, we would like to understand all critical points for this variational problem.

In order to state our main results, we will introduce some notation.  Given $n\geq 1$, we define the set of feasible constraints
\begin{equation*}
\F^n = \{ (e_1,\dots,e_n) \in \R^n : \con\neq\emptyset \}
\end{equation*}
which are attainable by some function in $H^n(\R)$.  We also define the set of constraints attainable by multisolitons of degree at most $N\geq 0$:
\begin{equation*}
\M^n_N = \big\{ (e_1,\dots,e_n) \in \R^n : \exists\ M\leq N\tx{ and }\bs{\beta},\bs{c}\in\R^M\tx{ with }Q_{\bs{\beta},\bs{c}}\in\con \big\} .
\end{equation*}

First we prove that when the constraints are attainable by a multisoliton of degree at most $n$, that multisoliton is the unique global minimizer:
\begin{thm}[Variational characterization]
\label{thm:var char}
Fix an integer $n\geq 1$.  Given constraints $(e_1,\dots,e_n) \in \M^n_n$, there exists a unique integer $N\leq n$ and parameters $\beta_1>\cdots>\beta_N>0$ so that the multisoliton $Q_{\bs{\beta},\bs{c}}$ lies in $\con$ for some (and hence all) $\bs{c}\in\R^N$.  Moreover, we have
\begin{equation*}
E_{n+1}(u) \geq E_{n+1}(Q_{\bs{\beta},\bs{c}}) \quad \tx{for all }u\in\con,
\end{equation*}
with equality if and only if $u=Q_{\bs{\beta},\bs{c}}$ for some $\bs{c}\in\R^N$.
\end{thm}

Together with an appropriate concentration compactness principle (\cref{thm:conc comp}) to analyze minimizing sequences, we also prove the orbital stability result of~\cite{Maddocks1993} via concentration compactness:
\begin{thm}[Orbital stability]
\label{thm:orb stab}
Fix an integer $n\geq 1$ and distinct positive parameters $\beta_1,\dots,\beta_n$.  Given $\eps>0$ there exists $\del>0$ so that for every initial data $u(0) \in H^n(\R)$ satisfying
\begin{equation*}
\inf_{\bs{c}\in\R^n}\, \snorm{ u(0) - Q_{\bs{\beta},\bs{c}} }_{H^n} < \del ,
\end{equation*}
the corresponding solution $u(t)$ of KdV~\eqref{eq:kdv} satisfies
\begin{equation*}
\sup_{t\in\R}\, \inf_{\bs{c}\in\R^n}\, \snorm{ u(t) - Q_{\bs{\beta},\bs{c}} }_{H^n} < \eps .
\end{equation*}
\end{thm}

It is now well-known that KdV is well-posed for initial data $u(0) \in H^n(\R)$ with $n\geq 1$~\cite{Kenig1991}.  Nevertheless, the $n=1$ case~\cite{Benjamin1972} of \cref{thm:orb stab} came nearly two decades before the corresponding well-posedness result~\cite{Kenig1991}.  This is because well-posedness plays only a small role in the proof of \cref{thm:orb stab}.  Indeed, the crux of the problem is to prove \cref{thm:orb stab} for Schwartz solutions, and the result for $H^n(\R)$ solutions then follows immediately once well-posedness in $H^n(\R)$ is known.

Many refinements have been made since the discovery of \cref{thm:orb stab}.  While we choose to work in the space $H^{n}(\R)$ because it is amenable to the functionals $E_1,\dots,E_{n+1}$, the regularity of \cref{thm:orb stab} in the scope of $H^s(\R)$ spaces has since been significantly lowered~\cites{Martel2002,Merle2003,Alejo2013,Killip2020}.  Moreover, the time-evolution of the parameter $\bs{c}$ has been shown to remain close to the usual evolution $c_j + 4\beta_j^2 t$ uniformly for $t>0$~\cite{Albert2007}.  In addition to Lyapunov stability statements like \cref{thm:orb stab}, the asymptotic stability of multisolitons has also been studied~\cites{Martel2002,Merle2003,Martel2005}.

Recently, the $n=2$ case of both \cref{thm:var char,thm:orb stab} was resolved by Albert and Nguyen~\cite{Albert2021}.  They also showed that for $n=1$ we have
\begin{equation*}
\M^1_1 = \{ e_1 : e_1 \geq 0 \} = \F^1 ,
\end{equation*}
but for $n=2$ we have
\begin{equation}
\begin{aligned}
\M^2_2
&= \big\{ (e_1,e_2) : e_1> 0,\ e_2 \in \big[ {- \tfrac{32}{5} (\tfrac{3}{8})^{\frac{5}{3}} e_1^{\frac{5}{3}}} ,\, {-2^{-\frac{2}{3}} \tfrac{32}{5} (\tfrac{3}{8})^{\frac{5}{3}} e_1^{\frac{5}{3}}} \big) \big\} \cup\{(0,0)\},\\
\F^2 
&= \big\{ (e_1,e_2) : e_1> 0,\ e_2 \geq - \tfrac{32}{5} (\tfrac{3}{8})^{\frac{5}{3}} e_1^{\frac{5}{3}} \big\}  \cup\{(0,0)\} .
\end{aligned}
\label{eq:M2 F2}
\end{equation}
These sets are depicted in \cref{fig:phase diagram}.  \Cref{thm:var char} says that for each $(e_1,e_2) \in \M^2_2$ the constrained minimizers of $E_3$ are multisolitons of degree at most $2$.  Moreover, by the $n=1$ case of \cref{thm:var char}, we know that for the constraints $e_1>0$ and $e_2 = - \tfrac{32}{5} (\tfrac{3}{8})^{\frac{5}{3}} e_1^{\frac{5}{3}}$ on the boundary of $\M^2_2$ the minimizer is a single soliton.  Likewise, in the case of $n=3$ constraints, the boundary $\M^3_3 \sm (\tx{int}\,\M^3_3)$ in $\R^3$ looks like the graph of a continuous function on $\M^2_2$, and so on.  In general, we will show that $\M^n_n$ is homeomorphic to the half-open simplex of parameters $\bs{\beta} \in \R^n$ corresponding to multisolitons of degree at most $n$ (cf.~\cref{thm:homeo}).

\begin{figure}[ht!]
\begin{center}
\includegraphics[width=0.5\linewidth]{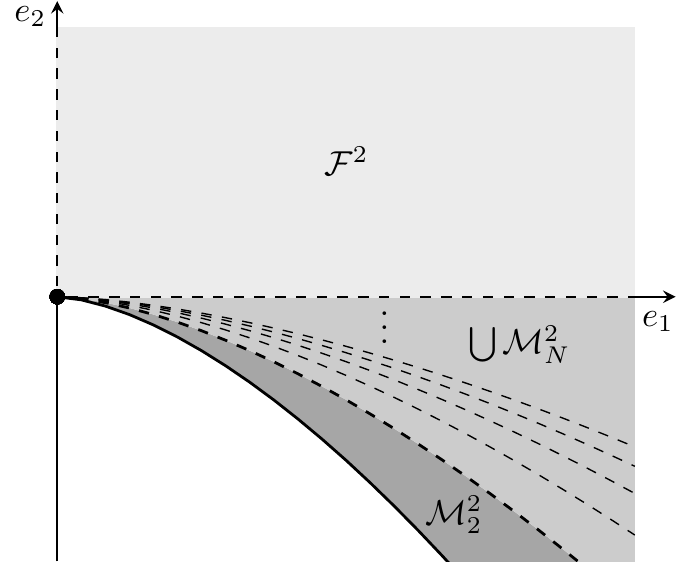}
\end{center}
\captionof{figure}{The sets~\eqref{eq:M2 F2} and~\eqref{eq:MN} of constraints.  The three shaded regions correspond to qualitatively different behavior for minimizing sequences.}
\label{fig:phase diagram}
\end{figure}

Albert and Nguyen's analysis does not easily extend to the general case however, because it makes crucial use of the fact that for $n=2$ all solutions of the Euler--Lagrange equation~\eqref{eq:novikov} are one- or two-solitons~\cite{Albert2017}.  Much is known about the ODEs~\eqref{eq:novikov}; they are completely integrable Hamiltonian systems and thus can be formally integrated.  Nevertheless, for $n\geq 3$ it is open whether all solutions to the Euler--Lagrange equation~\eqref{eq:novikov} are multisolitons of degree at most $n$.  (Specifically, it is difficult to show that if the Lagrange multipliers $\lambda_1,\dots,\lambda_n$ do not correspond to an $n$-soliton then solutions of~\eqref{eq:novikov} must be singular; see~\cite{Albert2017}*{\S6} for details.)

Another advantage of our method is that it enables us to study the variational problem and minimizing sequences even when the constraints $(e_1,\dots,e_n) \notin \M^n_n$ are not attainable by a multisoliton of degree at most $n$.  In order to present our results in this direction, we first need to recall some scattering theory.  Ever since the seminal work~\cite{Gardner1967}, scattering theory for the one-dimensional Schr\"odinger operator $-\partial^2_x + u$ with potential given by the solution $u(x) = u(t,x)$ to KdV for fixed $t$ has proved to be invaluable in our understanding of the dynamics of KdV.  Given a potential $u(x)$ in the weighted Lebesgue space 
\begin{equation*}
L^1_2(\R) = \bigg\{ f:\R\to\R \tx{ such that }\int_{-\infty}^\infty |f(x)|(1+|x|^2)\dx < \infty \bigg\} ,
\end{equation*}
the operator $-\partial^2_x + u$ on $L^2(\R)$ has purely absolutely continuous spectrum $[0,\infty)$ and finitely many simple negative eigenvalues $-\beta_1^2,\dots,-\beta_N^2$.  For such $u$ we can define the transmission and reflection coefficients at frequencies $k\in\R$, both of which are bounded by 1 in modulus.  Multisolitons are distinguished in this class by having identically vanishing reflection coefficients.  Additionally, the transmission coefficient $T(k;u)$ extends to a meromorphic function of $k$ in the upper half-plane $\C^+$ that is continuous down to $\R$ and whose only singularities are simple poles at the square roots $i\beta_1,\dots,i\beta_N$ of the eigenvalues.  Analytically, it is more convenient to work with the reciprocal $a(k;u)$ of the transmission coefficient; this function is holomorphic in $\C^+$ with simple zeros at $i\beta_1,\dots,i\beta_N$.

The Zakharov--Faddeev trace formulas~\cite{Zaharov1971} for the polynomial conserved quantities $E_n(u)$ generalize the formula~\eqref{eq:en multisoliton} for $u = Q_{\bs{\beta},\bs{c}}$.  Explicitly, this sequence begins
\begin{equation}
\begin{aligned}
E_1(u) &= \tfrac{4}{\pi} \int_{-\infty}^\infty k^2\log |a(k;u)| \dk + \tfrac{8}{3} \sum_{m=1}^N \beta_m^3 , \\
E_2(u) &= \tfrac{16}{\pi} \int_{-\infty}^\infty k^4\log |a(k;u)| \dk - \tfrac{32}{5} \sum_{m=1}^N \beta_m^5 , \\
E_3(u) &= \tfrac{64}{\pi} \int_{-\infty}^\infty k^6\log |a(k;u)| \dk + \tfrac{128}{7} \sum_{m=1}^N \beta_m^7 .
\end{aligned}
\label{eq:e1 e3 trace}
\end{equation}
These are valid for $u$ Schwartz, where we have both the energies $E_n(u)$ and the scattering data $a(k;u)$ and $\beta_1,\dots,\beta_N$ at our disposal.  The first terms on the RHS are the even moments of the nonnegative measure $\log|a(k;u)|\dk$ on $\R$.  The second terms are the odd moments of the positive numbers $\beta_1,\dots,\beta_N$, alternating in sign.  For multisolitons $q$ we have $|a(k;q)| \equiv 1$ on $\R$, and so the $\log|a(k;u)|$ moments vanish and we recover~\eqref{eq:en multisoliton}.

After obtaining the $n=2$ case of \cref{thm:var char,thm:orb stab}, Albert and Nguyen~\cite{Albert2021} made the reasonable conjecture for the remaining case $(e_1,e_2) \in \F^2\sm \M^2_2$ that minimizing sequences resemble a collection of solitons with at least two $\beta$ parameters equal.  Using our methods, we prove that this is partially true provided that the constraints are attainable by some multisoliton:
\begin{thm}
\label{thm:N gas}
Given constraints $(e_1,\dots,e_n) \in \M^n_{N} \sm \M^n_{N-1}$ for some $N \geq n+1$, the infimum of $E_{n+1}(u)$ over $u\in\con$ is finite but not attained.

Moreover, if $\{q_k\}_{k\geq 1} \subset \con$ is a minimizing sequence:
\begin{equation*}
E_1(q_k) \to e_1 , \quad\dots,\quad E_n(q_k) \to e_n , \quad E_{n+1}(q_k) \to \inf_{u\in\con} E_{n+1}(u) \quad\tx{as }k\to\infty ,
\end{equation*}
then there exist parameters $\bs{\beta}^1,\dots,\bs{\beta}^J$ of total degree $\sum_{j=1}^J \# \bs{\beta}_{n}^j = N$ taking at most $n$ distinct values so that along a subsequence we have
\begin{equation*}
\inf_{\bs{c}^1,\dots,\bs{c}^J}\, \bigg\lVert q_k - \sum_{j=1}^J Q_{\bs{\beta}^j,\bs{c}^j} \bigg\rVert_{H^n} \to 0 \quad\tx{as }k\to\infty. 
\end{equation*}
\end{thm}

We must have multiple multisolitons in the conclusion of \cref{thm:N gas}, because two multisolitons with a common $\beta$ parameter necessarily become infinitely separated as $k\to\infty$.  On the other hand, we cannot guarantee $N$ single-soliton profiles as originally conjectured in~\cite{Albert2021}, because two distinct values of the minimizing parameters $\bs{\beta}$ can correspond to a multisoliton that does not resemble well-separated solitons.

\Cref{thm:N gas} still does not account for all of the remaining feasible constraints $\F^n$!  In general, we can compute the boundary of $\M^n_N$ by finding the extrema of $E_{n+1}$ for $\beta_1,\dots,\beta_N>0$ distinct.  For $n=2$, it is not difficult to show that
\begin{equation*}
\M^2_N = \big\{ (e_1,e_2) : e_1> 0,\ e_2 \in \big[ {- \tfrac{32}{5} (\tfrac{3}{8})^{\frac{5}{3}} e_1^{\frac{5}{3}}} ,\, {-N^{-\frac{2}{3}} \tfrac{32}{5} (\tfrac{3}{8})^{\frac{5}{3}} e_1^{\frac{5}{3}}} \big) \big\} \cup\{(0,0)\}
\end{equation*}
for all $N\geq 2$.  Indeed, this is closely related to the elementary inequality
\begin{equation*}
N^{-\frac{2}{3}} \bigg( \sum_{m=1}^N \beta_m^3 \bigg)^{\frac{5}{3}}
\leq \sum_{m=1}^N \beta_m^5
\leq \bigg( \sum_{m=1}^N \beta_m^3 \bigg)^{\frac{5}{3}} 
\end{equation*}
which expresses the equivalence of the $\ell^3$- and $\ell^5$-norms on $\R^N$.  The sets $\M^2_N$ are depicted in the phase diagram of \cref{fig:phase diagram}, and can also be understood in terms of the single dimensionless variable $e_2 e_1^{-5/3}$.  Note that this implies that the set $\M^n_N\sm\M^n_{N-1}$ in \cref{thm:N gas} is nonempty for all $n\geq 2$ and $N\geq 1$, since increasing $N$ always introduces new values of $e_2$.  Moreover, we see that the set
\begin{equation}
\bigcup_{N\geq 0} \M^2_N = \big\{ (e_1,e_2) : e_1> 0,\ e_2 \in \big[ {- \tfrac{32}{5} (\tfrac{3}{8})^{\frac{5}{3}} e_1^{\frac{5}{3}}} ,\, 0 \big) \big\} \cup\{(0,0)\}
\label{eq:MN}
\end{equation}
misses a large portion of the feasible constraints $\F^2$ in~\eqref{eq:M2 F2}.

We discover that in the remaining case where the constraints $(e_1,\dots,e_n)\in \F^n \sm \bigcup_{N\geq 0} \M^n_N$ are not attainable by any multisoliton, Albert--Nguyen's conjecture cannot be true and entirely different behavior is exhibited.  To illustrate this point, we characterize Schwartz minimizing sequences in the case $n=2$:
\begin{thm}
\label{thm:point mass}
Given constraints $(e_1,e_2) \in \F^2$ with $(e_1,e_2) \notin \M^2_N$ for all $N$, the infimum of $E_{3}(u)$ over $u\in\con\cap\schwar(\R)$ is finite but not attained.

Moreover, if $\{q_j\}_{j\geq 1} \subset \con\cap\schwar(\R)$ is a minimizing sequence:
\begin{equation*}
E_1(q_j) \to e_1 , \quad E_2(q_j) \to e_2 , \quad E_{3}(q_j) \to \inf_{u\in\con} E_{3}(u)
\quad\tx{as }j\to\infty ,
\end{equation*}
then $\beta_{j,m} \to 0$ as $j\to\infty$ for all $m$, and $k\mapsto \log|a(k;q_j)|$ converges in the sense of distributions to the even extension of a unique Dirac delta distribution (i.e.~$c_0(\dd\del_{k_0}(k) + \dd\del_{-k_0}(k))$ for unique constants $c_0,k_0\geq 0$).
\end{thm}

What might such a minimizing sequence look like?  We can exhibit a family of these sequences by means of an ansatz inspired by the Wigner--von Neumann example of a Schr\"odinger potential with a positive eigenvalue.  Given parameters $c>0$ and $k\geq 0$, it is straightforward to check that the sequence
\begin{equation*}
q_n(x) = \sqrt{\tfrac{c}{n}} e^{-\frac{x^2}{2n^2}} \cos(2kx)
\end{equation*}
obeys
\begin{equation*}
E_1(q_n) \to \tfrac{\sqrt{\pi}}{4} c , \quad
E_2(q_n) \to \sqrt{\pi} ck^2 , \quad
E_3(q_n) \to 4\sqrt{\pi} ck^4
\end{equation*}
as $n\to\infty$.  In the proof of \cref{thm:point mass} we will explicitly compute the constrained infimum of $E_3$, and it is given by $e^2_2e_1^{-1}$ (cf.~\cref{thm:point mass lem 2}).  We see that the limit of $E_3(q_n)$ above is exactly equal to this quantity, and so $\{q_n\}_{n\geq 1}$ is a Schwartz minimizing sequence.  By \cref{thm:point mass}, we deduce that this sequence has vanishing $\beta$ parameters and $\log|a(k;q_n)|$ converging to the even extension of a delta distribution.

This paper is organized as follows. In \cref{sec:prelim} we recall some scattering theory and facts about the energy functionals $E_n$, with the Zakharov–-Faddeev trace formulas~\eqref{eq:en trace} lying at the center.  In \cref{sec:eng}, we then further analyze the functionals $E_1,\dots,E_{n+1}$ on the manifold of multisolitons and use this to describe the set $\M^n_n$ of constraints.

The proof of \cref{thm:var char} is then presented in \cref{sec:min}.  A key step is realizing that in order to minimize $E_{n+1}$, a minimizer $q$ must satisfy $\log|a(k;q)| \equiv 0$ on $\R$ (cf.~\eqref{eq:trans coeff 1} and~\eqref{eq:en trace}).  We know that $\log|a(k;u)|$ can be brought all the way down to zero since the constraints can be met solely by the moments of $\bs{\beta}\in\R^n$; this is why we must assume that $(e_1,\dots,e_n)\in \M_n^n$ in \cref{thm:var char}.   Next, we prove that $|a(k;q)|\equiv 1$ on $\R$ if and only if $q$ is a multisoliton.  The ``if'' statement is already known from scattering theory (cf.~\eqref{eq:blaschke}).  For the reverse implication, we use some classical complex analysis to characterize $k\mapsto a(k;q)$ on $\C^+$ and conclude that $q$ is a multisoliton (cf.~\cref{thm:inner outer fac}).   It then remains to show that on the manifold of multisolitons, there is a unique minimizing set of $\beta$ parameters.  First, we can rule out the case of $N\geq n+1$ parameters by a variational argument (cf.~\cref{thm:beta mom wiggle}).  For the remaining $N\leq n$ unknown parameters $\beta_1,\dots,\beta_N$, the formulas~\eqref{eq:en multisoliton} for the $n$ constraints provide a nonlinear system of equations, which we show has a unique solution in \cref{thm:unique params}.

In \cref{sec:orb stab} we apply a concentration compactness principle (\cref{thm:conc comp}) to minimizing sequences in order to prove \cref{thm:orb stab}.

In \cref{sec:N gas}, we prove \cref{thm:N gas} by adapting the methods of \cref{sec:min,sec:eng,sec:orb stab} to allow for non-distinct $\beta$ parameters.

Finally, in \cref{sec:point mass} we prove \cref{thm:point mass}.  The proof is again based on the trace formulas~\eqref{eq:e1 e3 trace}.  The condition $(e_1,e_2) \notin \M^2_N$ requires that the $\log|a(k;q_j)|$ moments are nonvanishing as $j\to\infty$.  Consequently, the sequence $\{q_j\}_{j\geq 1}$ is minimizing a constrained moment problem for measures, and such minimizers are finite linear combinations of point masses.  This particular moment problem for $n=2$ is easily solved using the Cauchy--Schwarz inequality, but for general $n$ it is a Stieltjes moment problem.  (We recommend~\cite{Simon1998} for an introduction to this classical analysis result.)

\begin{ack}
I was supported in part by NSF grants DMS-1856755 and DMS-1763074.  I would also like to thank my advisors, Rowan Killip and Monica Vi\c{s}an, for their guidance, and Jaume de Dios Pont for pointing out the proof method from~\cite{Steinig1971} for \cref{thm:unique params}.
\end{ack}


\section{Preliminaries}
\label{sec:prelim}

In this section, we recall some facts about the energy functionals $E_n$ and some results from scattering theory for future reference.  In particular, this will enable us to formulate the Zakharov--Faddeev trace formulas~\eqref{eq:en trace} that lie at the heart of our analysis.

The functionals~\eqref{eq:e1 and e2} and~\eqref{eq:e3} are the beginning of an infinite sequence of polynomial conserved quantities~\eqref{eq:en 3}.  We will only need certain properties of these functionals however, rather than their exact formula.  
\begin{prop}[\cite{Miura1968}]
\label{thm:en}
Given an integer $n\geq 0$, there exists a functional of the form
\begin{equation}
E_{n+1}(u) = \int_{-\infty}^\infty \big[ \tfrac{1}{2} \big( u^{(n)} \big)^2 + P_{n+1}(u) \big]\dx
\label{eq:en}
\end{equation}
that is conserved for Schwartz solutions $u(t)$ to the KdV equation~\eqref{eq:kdv}, where $P_{n+1}$ is a polynomial in $u,u',\dots,u^{(n-1)}$.  Moreover, each term of $P_{n+1}$ is of the form $c_{\alpha_1,\dots,\alpha_d} u^{(\alpha_1)}\cdots u^{(\alpha_d)}$ with $d\geq 3$ and obeys
\begin{equation}
\sum_{j=1}^d \alpha_j = 2n + 4 - 2d
\quad\tx{and}\quad
0 \leq \alpha_j \leq n-1 .
\label{eq:en 2}
\end{equation}
\end{prop}

Each term of $P_{n+1}$ is of cubic or higher order, and the condition~\eqref{eq:en 2} simply says that they share the same scaling symmetry as the quadratic term $\tfrac{1}{2} ( u^{(n)} )^2$.  In particular, this requires that the degree of $P_{n+1}$ is at most $n+2$.

When combined with Sobolev embedding, a classical argument (cf.~\cite{Lax1975}*{Th.~3.1} in the periodic case) yields estimates for the functionals~\eqref{eq:en}:
\begin{lem}
Given $n\geq 0$, we have
\begin{equation}
E_{n+1}(u) \lesssim_{\norm{u}_{H^n}} 1
\quad\tx{and}\quad
\norm{u}_{H^n} \lesssim_{E_1(u),\dots,E_{n+1}(u)} 1
\label{eq:ests for en cts}
\end{equation}
uniformly for $u\in\schwar(\R)$.  Moreover, $E_{n+1} : H^n(\R) \to \R$ is continuous.
\end{lem}

Here (and throughout this paper), we are using the familiar $L^2$-based Sobolev spaces $H^s(\R)$ (and the $L^p$-based spaces $W^{j,p}(\R)$) of real-valued functions on $\R$.  In addition to these classes, the scattering theory that we need to state our trace formulas will require that we work in the weighted $L^1$-spaces
\begin{equation*}
L^1_j(\R) := \bigg\{ f:\R\to\R \tx{ such that } \int_{-\infty}^\infty |f(x)| (1+|x|^j)\dx < \infty \bigg\}
\end{equation*}
with $j\geq 1$.  When we need a common ground, we will use the class $\schwar(\R)$ of Schwartz functions.

Given a potential $q\in L^1_1(\R)$ and $k\in\R\sm\{0\}$, the Jost functions $f_j(x;k)$ are the unique solutions to the corresponding Schr\"odinger equation
\begin{equation*}
- f''_j + qf_j = k^2 f_j , \quad j=1,2
\end{equation*}
with the asymptotics
\begin{equation*}
f_1(x;k) \sim e^{ikx} \quad\tx{as } x\to +\infty , \qquad
f_2(x;k) \sim e^{-ikx} \quad\tx{as } x\to -\infty .
\end{equation*}
The transmission and reflection coefficients $T_j(k)$ and $R_j(k)$ are then uniquely determined by
\begin{align*}
T_1(k) f_2(x;k) &= R_1(k) f_1(x;k) + f_1(x;-k) , \\
T_2(k) f_1(x;k) &= R_2(k) f_2(x;k) + f_2(x;-k) .
\end{align*}

Forward scattering theory tells us that the transmission and reflection coefficients satisfy the following properties.  Proofs of these facts can be found in many introductory texts on the subject; however, we recommend the paper~\cite{Deift1979}*{\S2 Th.~1} of Deift and Trubowitz for a complete and self-contained proof.
\begin{prop}[Forward scattering theory]
\label{thm:scat thy}
If $q\in L^1_2(\R)$, then the scattering matrix
\begin{equation*}
S(k) := \begin{pmatrix}
T_1(k) & R_2(k) \\ R_1(k) & T_2(k) 
\end{pmatrix}
\end{equation*}
extends to a continuous function of $k\in\R$ satisfying the following properties:
\begin{enumerate}[label=(\roman*)]
\item (Symmetry) For all $k\in\R$,
\begin{equation*}
T_1(k) \equiv T_2(k) =: T(k) .
\end{equation*}
\item (Unitarity) The matrix $S(k)$ is unitary for all $k\in \R$:
\begin{equation*}
T(k) \ol{R_2(k)} + R_1(k) \ol{T(k)} \equiv 0 , \qquad
|T(k)|^2 + |R_j(k)|^2 \equiv 1 \quad\tx{for } j=1,2.
\end{equation*}
\item (Analyticity) $T(k)$ is meromorphic in the open upper half-plane $\C^+$ and is continuous down to $\R$.  Moreover, $T(k)$ has a finite number of poles $i\beta_1,\dots,i\beta_N$, all of which are simple and on the imaginary axis, and $-\beta_1^2$, $\dots,-\beta_N^2$ are the bound states of the operator $-\partial^2_x +q$.
\item (Asymptotics)  We have
\begin{align*}
T(k) &= 1 + O(\tfrac{1}{k}) \quad\tx{as } |k|\to\infty \tx{ uniformly for } \Im k \geq 0 , \\
R_j(k) &= O(\tfrac{1}{k}) \quad\tx{as } |k|\to\infty , \ k\in\R .
\end{align*}
\item (Rate at $k=0$)  $T(k)$ is nonvanishing for $k\in\ol{\C^+}\sm\{0\}$, and either
\begin{enumerate}
\item $|T(k)| \geq c>0$ for all $k\in\ol{\C^+}$, or
\item $T(k) = T'(0)k + o(k)$ for $k\in\ol{\C^+}$ with $T'(0)\neq 0$ and $R_j(k) = -1 + R'_j(0) k + o(k)$ for $k\in\R$.
\end{enumerate}
\item (Reality) For all $k\in\R$,
\begin{equation*}
\ol{T(k)} = T(-k) , \quad \ol{R_j(k)} = R_j(-k) \quad\tx{for } j=1,2 .
\end{equation*}
\end{enumerate}
\end{prop}

Our trace formulas are most conveniently stated in terms of the reciprocal of the transmission coefficient:
\begin{equation*}
a(k;q) := \tfrac{1}{T(k)} .
\end{equation*}
For $q\in L^1_2(\R)$, \cref{thm:scat thy} tells us that $k \mapsto a(k;q)$ is a holomorphic function on the open upper-half plane $\C^+$ and is continuous down to $\R$.  It has finitely many zeros $i\beta_1,\dots,i\beta_N$ in $\C^+$, all of which are simple and on the imaginary axis.  Moreover, we have the boundary conditions
\begin{gather}
|a(k;q)| \geq 1 \quad\tx{for all } k\in\R , 
\label{eq:trans coeff 1} \\
|a(k;q) - 1| = O(\tfrac{1}{|k|}) \quad\tx{as }|k|\to\infty\tx{ uniformly for }\Im k\geq 0 ,
\label{eq:trans coeff 2}
\end{gather}
along with the reality condition
\begin{equation}
\ol{ a(k;q) } = a(-\ol{k};q) \quad\tx{for all } k\in \C^+ .
\label{eq:trans coeff 3}
\end{equation}

For $u\in\schwar(\R)$, the Zakharov--Faddeev trace formulas~\cite{Zaharov1971} provide an alternative representation of the polynomial conserved quantities:
\begin{equation}
E_n(u) = \tfrac{2^{2n}}{\pi} \int_{-\infty}^\infty k^{2n} \log|a(k;u)|\dk + (-1)^{n+1} \tfrac{2^{2n+1}}{2n+1} \sum_{m=1}^N \beta_m^{2n+1} .
\label{eq:en trace}
\end{equation}
The measure $\log|a(k;u)|\dk$ on $\R$ is nonnegative and even, and the first terms on the RHS are its even moments (starting with the second), which are finite for $u\in\schwar(\R)$~\cite{Zaharov1971}.  The second terms are the odd moments of the distinct positive numbers $\beta_1,\dots,\beta_N$ (starting with the third) and are alternating in sign. 

Later, we will deduce that a constrained minimizer $q$ of $E_{n+1}$ must have certain scattering data due to the trace formulas~\eqref{eq:en trace}.  Consequently, it will be useful to know when we can reconstruct the potential $q$ from the scattering data~\cite{Deift1979}*{\S5 Th.~3}:
\begin{prop}[Inverse scattering theory]
\label{thm:inv scat thy}
A matrix 
\begin{equation*}
S(k) := \begin{pmatrix}
T_1(k) & R_2(k) \\ R_1(k) & T_2(k) 
\end{pmatrix}, \quad k\in\R
\end{equation*}
is the scattering matrix for some $q\in L^1_2(\R)$ without bound states if and only if
\begin{enumerate}[label=(\roman*)]
\item (Symmetry) For all $k\in\R$,
\begin{equation*}
T_1(k) \equiv T_2(k) =: T(k) .
\end{equation*}
\item (Unitarity) The matrix $S(k)$ is unitary for all $k\in\R$:
\begin{equation*}
T(k) \ol{R_2(k)} + R_1(k) \ol{T(k)} \equiv 0 , \qquad
|T(k)|^2 + |R_j(k)|^2 \equiv 1 \quad\tx{for } j=1,2.
\end{equation*}
\item (Analyticity) $T(k)$ is analytic in the open upper half-plane $\C^+$ and is continuous down to $\R$.
\item (Asymptotics)  We have
\begin{align*}
T(k) &= 1 + O(\tfrac{1}{k}) \quad\tx{as } |k|\to\infty \tx{ uniformly for } \Im k \geq 0 , \\
R_j(k) &= O(\tfrac{1}{k}) \quad\tx{as } |k|\to\infty , \ k\in\R .
\end{align*}
\item (Rate at $k=0$)  $T(k)$ is nonvanishing for $k\in\ol{\C^+}\sm\{0\}$, and either
\begin{enumerate}
	\item $|T(k)| \geq c>0$ for all $k\in\ol{\C^+}$, or
	\item $T(k) = T'(0)k + o(k)$ for $k\in\ol{\C^+}$ with $T'(0)\neq 0$ and $R_j(k) = -1 + R'_j(0) k + o(k)$ for $k\in\R$.
\end{enumerate}
\item (Reality) For all $k\in\R$,
\begin{equation*}
\ol{T(k)} = T(-k) , \quad \ol{R_j(k)} = R_j(-k) \quad\tx{for } j=1,2 .
\end{equation*}
\item (Fourier decay) The Fourier transforms $F_j := \wh{R}_j$, $j=1,2$ are absolutely continuous and
\begin{equation*}
\int_{-\infty}^a |F'_1(\kappa)|(1+\kappa^2)\,d\kappa <\infty \quad\tx{and}\quad 
\int_{a}^\infty |F'_2(\kappa)|(1+\kappa^2)\,d\kappa <\infty
\end{equation*}
for all $a\in\R$.
\end{enumerate}
\end{prop}

The characterization in \cref{thm:inv scat thy} is most easily stated in terms of potentials $q$ without bound states.  This does not pose a problem though, because there is an explicit formula for modifying $q$ in order to prescribe any number of bound states~\cite{Deift1979}*{\S3 Th.~6}.  Rather than the explicit formula for $q$, we will only need to keep track of the changes in the transmission coefficient:
\begin{prop}[Adding bound states]
\label{thm:inv scat thy 2}
Let $q(x)\in L^1_1(\R)$ be a potential without bound states and $\beta_1,\dots,\beta_N >0$ distinct.  Then there exists a potential $q(x;+N) \in L^1_1(\R)$ with the $N$ bound states $-\beta_1^2 , \dots, - \beta_N^2$.  Moreover, the transmission coefficient is related to that of $q(x)$ via
\begin{equation}
T(k;+N) = T(k) \prod_{m=1}^N \frac{k+i\beta_m}{k-i\beta_m} .
\label{eq:blaschke 2}
\end{equation}
\end{prop}

Within this narrative, multisolitons are characterized by having vanishing reflection coefficients.  In view of the preceding, this identifies the formula for $a(k;q)$:
\begin{cor}[Characterization of multisolitons]
Given distinct $\beta_1,\dots,\beta_N>0$ and $q\in\schwar(\R)$, we have $q=Q_{\bs{\beta},\bs{c}}$ for some $\bs{c}\in\R^N$ if and only if
\begin{equation}
a(k;q) = \prod_{m=1}^N \frac{k-i\beta_m}{k+i\beta_m} .
\label{eq:blaschke}
\end{equation}
\end{cor}

Notice that the formula~\eqref{eq:blaschke} is a finite Blaschke product from $\C^+$ to the unit disk, with zeros that are distinct and lie only on the imaginary axis.  In particular, we see that multisolitons $Q_{\bs{\beta},\bs{c}}$ satisfy $|a(k;Q_{\bs{\beta},\bs{c}})| \equiv 1$ on $\R$, and so the $\log|a|$ moments in the trace formulas~\eqref{eq:en trace} vanish.  Consequently, the functionals $E_n$ for a multisoliton $Q_{\bs{\beta},\bs{c}}$ are independent of $\bs{c}$.


\section{The polynomial conserved quantities}
\label{sec:eng}

The purpose of this section is to analyze the functionals $E_1,\dots,E_{n+1}$ restricted to the manifold of multisolitons, and to then use this to describe the set $\M^n_n$ of constraints.

First, we will prove that as long as we have at least $n+1$ distinct positive $\beta$ parameters at our disposal, we can decrease $E_{n+1}$ while preserving $E_1,\dots,E_n$.  This surprisingly powerful fact will turn out to be an integral part of our analysis.
\begin{lem}
\label{thm:beta mom wiggle}
Fix $n\geq 1$, and suppose $Q_{\bs{\beta},\bs{c}}$ is a multisoliton of degree $N \geq n+1$.  Then there exist $\wt{\beta}_1,\dots,\wt{\beta}_N > 0$ distinct so that
\begin{equation*}
E_m(Q_{\wt{\bs{\beta}},\bs{c}}) = E_m(Q_{\bs{\beta},\bs{c}}) \quad \tx{for } m=1,\dots,n , \quad\tx{but}\quad
E_{n+1}(Q_{\wt{\bs{\beta}},\bs{c}}) < E_{n+1}(Q_{\bs{\beta},\bs{c}}) .
\end{equation*}
\end{lem}
\begin{proof}
We will apply the implicit function theorem to the first $n+1$ trace formulas~\eqref{eq:en trace} as functions of $\bs{\beta}$.  Reorder $\bs{\beta}$ so that $\beta_1>\cdots>\beta_N>0$.  Define the function $f:\R^{n+1}\to\R^n$ by
\begin{equation}
f(x_1,\dots,x_{n+1})
= \begin{pmatrix}
x_1^3 + x_2^3 + \dots + x_{n+1}^3 \\
x_1^5 + x_2^5 + \dots + x_{n+1}^5 \\
\vdots \\
x_1^{2n+1} + x_2^{2n+1} + \dots + x_{n+1}^{2n+1}
\end{pmatrix} .
\label{eq:beta mom wiggle 6}
\end{equation} 
This function has derivative matrix
\begin{equation}
Df(\beta_1,\dots,\beta_{n+1})
= \left( \!\begin{array}{ccc;{4pt/2pt}c}
3\beta_1^2 & \!\dots\! & 3\beta_n^2 & 3\beta_{n+1}^2 \\
5\beta_1^4 & \!\dots\!  & 5\beta_n^4 & 5\beta_{n+1}^4 \\
\vdots & \!\ddots\! & \!\vdots\!  & \vdots \\
(2n+1)\beta_1^{2n} & \!\dots\!  & (2n+1)\beta_n^{2n} & (2n+1) \beta_{n+1}^{2n}
\end{array} \!\right) .
\label{eq:beta mom wiggle 7}
\end{equation}
The left $n\times n$ block matrix is a Vandermonde matrix after pulling out common factors from each row and column, and thus has determinant
\begin{equation}
3\cdot 5 \cdots (2n+1) \, \beta_1^2 \cdots \beta_n^2\, \prod_{j<k} (\beta_k^2 - \beta_j^2) .
\label{eq:beta mom wiggle 2}
\end{equation}
This is nonvanishing as $\beta_1,\dots,\beta_{n}$ are positive and distinct.  The implicit function theorem then implies that there exists $\eps>0$ and $C^1$ functions $x_1(x_{n+1})$, $\dots$, $x_n(x_{n+1})$ defined on $(\beta_{n+1}-\eps,\beta_{n+1}+\eps)$ so that
\begin{equation}
f(x_1(x_{n+1}),\dots,x_n(x_{n+1}),x_{n+1}) \equiv \begin{pmatrix}
\beta_1^3 + \dots + \beta_{n+1}^3 \\
\beta_1^5 + \dots + \beta_{n+1}^5 \\
\vdots \\
\beta_1^{2n+1} + \dots + \beta_{n+1}^{2n+1}
\end{pmatrix} 
\label{eq:beta mom wiggle 1}
\end{equation}
for $x_{n+1} \in (\beta_{n+1}-\eps,\beta_{n+1}+\eps)$.

It remains to show that we can pick $x_{n+1}$ in a way that decreases the next odd moment.  To this end, we will compute its derivative at $x_{n+1} = \beta_{n+1}$:
\begin{equation}
\begin{aligned}
&\frac{d}{dx_{n+1}} \bigg|_{x_{n+1}=\beta_{n+1}} \big[ x_1(x_{n+1})^{2n+3} + \dots + x_n(x_{n+1})^{2n+3} + x_{n+1}^{2n+3} \big] \\
&= (2n+3) \left[ 
\begin{pmatrix} \beta_1^{2n+2} & \!\dots\! & \beta_n^{2n+2} \end{pmatrix}
\begin{pmatrix} x'_1(\beta_{n+1}) \\ \vdots \\ x'_n(\beta_{n+1}) \end{pmatrix}
+ \beta_{n+1}^{2n+2} \right] .
\end{aligned}
\label{eq:beta mom wiggle 3}
\end{equation}
The derivative of $x_1(x_{n+1})$, $\dots$, $x_n(x_{n+1})$ is determined by differentiating~\eqref{eq:beta mom wiggle 1} at $x_{n+1}=\beta_{n+1}$.  This yields
\begin{equation}
\begin{pmatrix} x'_1(\beta_{n+1}) \\ \vdots \\ x'_n(\beta_{n+1}) \end{pmatrix}
= - \begin{pmatrix} 3\beta_1^2 & \!\dots\! & 3\beta_n^2 \\ \vdots & \!\ddots\! & \vdots \\ (2n+1)\beta_1^{2n} & \!\dots\! & (2n+1)\beta_n^{2n} \end{pmatrix}^{\!-1} \!\!
\begin{pmatrix} 3\beta_{n+1}^2 \\ \vdots \\ (2n+1)\beta_{n+1}^{2n} \end{pmatrix} .
\label{eq:beta mom wiggle 4}
\end{equation}
Inserting this into~\eqref{eq:beta mom wiggle 3}, we obtain an expression solely in terms of $\beta_1,\dots,\beta_{n+1}$.  In order to compute this, we will leverage that it is a Schur complement for the derivative matrix~\eqref{eq:beta mom wiggle 7} but with an appended row.  Specifically, if we define the $(n+1)\times (n+1)$ block matrix
\begin{equation*}
\begin{pmatrix}
A & b \\ c & d
\end{pmatrix}
= \left( \!\begin{array}{ccc;{4pt/2pt}c}
3\beta_1^2 & \!\dots\! & 3\beta_n^2 & 3\beta_{n+1}^2 \\
5\beta_1^4 & \!\dots\!  & 5\beta_n^4 & 5\beta_{n+1}^4 \\
\vdots & \!\ddots\! & \!\vdots\!  & \vdots \rule[-0.9ex]{0pt}{0pt} \\ \hdashline[4pt/2pt]
(2n+3)\beta_1^{2n+2} & \!\dots\!  & (2n+3)\beta_n^{2n+2} & (2n+3) \beta_{n+1}^{2n+2} \rule{0pt}{2.6ex}
\end{array} \!\right) ,
\end{equation*}
then LHS\eqref{eq:beta mom wiggle 3} is now given by
\begin{equation*}
d - cA^{-1}b .
\end{equation*}
On the other hand, applying one step of Gaussian elimination to our block matrix yields
\begin{equation*}
\begin{pmatrix}
A & b \\ c & d
\end{pmatrix}
= \begin{pmatrix}
I & 0 \\ cA^{-1} & 1
\end{pmatrix} \begin{pmatrix}
A & b \\ 0 & d - cA^{-1}b
\end{pmatrix} .
\end{equation*}
Taking the determinant of both sides, we deduce that LHS\eqref{eq:beta mom wiggle 3} is equal to
\begin{equation*}
\det \begin{pmatrix}
A & b \\ c & d
\end{pmatrix} (\det A)^{-1} .
\end{equation*}
Both terms above can be computed by the Vandermonde determinant formula: the determinant of $A$ is given by~\eqref{eq:beta mom wiggle 2} for $n$ and the determinant of the block matrix is given by~\eqref{eq:beta mom wiggle 2} for $n+1$.

Altogether, we conclude
\begin{equation*}
\begin{aligned}
&\frac{d}{dx_{n+1}} \bigg|_{x_{n+1}=\beta_{n+1}} \big[ x_1(x_{n+1})^{2n+3} + \dots + x_n(x_{n+1})^{2n+3} + x_{n+1}^{2n+3} \big] \\
&= (2n+3) \beta_{n+1}^2 \prod_{j=1}^n (\beta_{n+1}^2 - \beta_j^2) .
\end{aligned}
\end{equation*}
The RHS is nonvanishing since $\beta_1,\dots,\beta_{n+1}$ are positive and distinct.  As $\beta_{1},\dots,\beta_{N}$ were distinct to begin with, we conclude that there exists $x_{n+1} \in (\beta_{n+1}-\eps,\beta_{n+1}+\eps)$ sufficiently close to $\beta_{n+1}$ so that the values $x_1(x_{n+1})$, $\dots$, $x_n(x_{n+1})$, $x_{n+1}$, and $\beta_{n+2}$ remain distinct, and
\begin{equation*}
(-1)^n \big[ x_1(x_{n+1})^{2n+3} + \dots + x_n(x_{n+1})^{2n+3} + x_{n+1}^{2n+3} \big]
\end{equation*}
is strictly less than its value at $x_{n+1} = \beta_{n+1}$.  This quantity is the value of $E_{n+1}$ for the multisoliton with $\beta$ parameters $x_1(x_{n+1})$, $\dots$, $x_n(x_{n+1})$, $x_{n+1}$.  Replacing $\beta_1,\dots,\beta_{n+1}$ by $x_1(x_{n+1})$, $\dots$, $x_n(x_{n+1})$, $x_{n+1}$ in $\bs{\beta}$, we obtain new distinct parameters $\wt{\beta}_1>\cdots>\wt{\beta}_N>0$ (with $\wt{\beta}_j = \beta_j$ for $j\geq n+2$) so that the multisoliton $Q_{\wt{\bs{\beta}},\bs{c}}$ decreases $E_{n+1}$ while preserving $E_1,\dots,E_n$.
\end{proof}

Our next step is to find the unique set of distinct $\beta$ parameters with at most $n$ values that attain the constraints.  As $(e_1,\dots,e_n)\in \M^n_n$, it only remains to show that there is at most one solution:
\begin{lem}
\label{thm:unique params}
Fix $n\geq 1$.  Given constraints $e_1,\dots,e_n$, there is at most one choice of $N\leq n$ and $\beta_1>\cdots>\beta_N>0$ so that
\begin{equation*}
E_m(Q_{\bs{\beta},\bs{c}}) = e_m \quad\tx{for }m=1,\dots,n \tx{ and any }\bs{c}\in\R^N.
\end{equation*}
\end{lem}

We will follow the clever argument from~\cite{Steinig1971}, which we learned about from~\cite{Drury1987}*{\S3}.  In fact, the result in~\cite{Steinig1971} is even more general: it is shown that any $n$ power sums of $n$ distinct positive real numbers has at most one solution (up to permutation), in addition to some generalizations.  However, we will provide a complete and self-contained proof here for future reference (in \cref{thm:unique params 2}).
\begin{proof}
Suppose towards a contradiction that there exist $\beta_1>\cdots>\beta_N>0$ and $\wt{\beta}_1 > \dots > \wt{\beta}_{\wt{N}} > 0$ with $\wt{N}\leq N \leq n$ such that $E_k(Q_{\bs{\beta},\bs{c}}) = E_k(Q_{\wt{\bs{\beta}},\bs{c}})$ for $k=1,\dots,n$.  By the trace formulas~\eqref{eq:en}, this requires that
\begin{equation}
\sum_{m=1}^N \beta_m^{2k+1} = \sum_{m=1}^{\wt{N}} \wt{\beta}_m^{2k+1} \quad\tx{for }k=1,\dots,n .
\label{eq:unique params 4}
\end{equation}
Consider the function $f:\R\to\R^n$ given by $f(x) = (x^3,x^5,\dots,x^{2n+1})$.  After canceling common terms and moving everything to the LHS, we obtain
\begin{equation*}
\sum_{j=1}^M \eps_j f(\beta_j) = 0
\end{equation*}
for some $\beta_1>\cdots>\beta_M>0$, $M\leq 2N$, and signs $\eps_j\in \{\pm 1\}$.

Next, we append $2N-M$ copies of $\beta_j :=0$ and $\eps_j := -1$ for $j=M+1,\dots,2N$ so that $\sum_{j=1}^{2N} \eps_j = 0$.  Using summation by parts, this allows us to write
\begin{equation*}
0
= \sum_{j=1}^M \eps_j f(\beta_j) + \sum_{j=M+1}^{2N} (-1) f(0)
= \sum_{j=1}^{2N} \eps_j f(\beta_j)
= \sum_{j=1}^{2N-1} \alpha_j [ f(\beta_j) - f(\beta_{j+1}) ] ,
\end{equation*}
where $\alpha_j = \sum_{k=1}^j \eps_k$.  By the fundamental theorem of calculus, we obtain
\begin{equation}
0 = \int_0^{\beta_1} \phi(s) f'(s)\ds
\label{eq:unique params 1}
\end{equation}
for the step function $\phi$ which takes value $\alpha_j$ on the interval $(\beta_{j+1},\beta_j)$.  Let $I_1,\dots,I_m$ denote the disjoint intervals in $[0,\beta_1]$ (in consecutive order) on which $\phi$ is nonvanishing and has constant sign.  Note that we can have at most $n$ such intervals; indeed, $j$ must increment by two in order for $\alpha_j$ to change sign, which together with the first and last indices $j$ account for all of the $2N\leq 2n$ parameters.  The equality~\eqref{eq:unique params 1} then tells us that the rows of the $m\times n$ matrix $A$ with entries
\begin{equation*}
a_{jk} = \int_{I_j} |\phi(s)| (f'(s))_k \ds , \quad j\in\{1,\dots,m\},\ k \in\{1,\dots,n\}
\end{equation*}
are linearly dependent.

We claim that $A$ is a strictly totally positive matrix---i.e.~all of the minors of $A$ are strictly positive---which will contradict the linear dependence among the rows.  Given two subsets of indices $J\subset \{1,\dots,m\}$ and $K\subset\{1,\dots,n\}$, we can write the corresponding minor of $A$ as
\begin{align*}
&\tx{minor}_{J,K}\begin{pmatrix}
\int_{I_1} |\phi(s)|(f'(s))_1\ds & \!\dots\! & \int_{I_1} |\phi(s)|(f'(s))_n\ds \\
\vdots & \!\ddots\! & \vdots \\
\int_{I_m} |\phi(s)|(f'(s))_1\ds & \!\dots\! & \int_{I_m} |\phi(s)|(f'(s))_n\ds
\end{pmatrix} \\
&= \int_{I_{j_1}} \!\!\!\cdots \int_{I_{j_{|J|}}} \!\!\tx{minor}_{J,K} \begin{pmatrix}
|\phi(s_1)|(f'(s_1))_1 & \!\dots\! & |\phi(s_1)|(f'(s_1))_n \\
\vdots & \!\ddots\! & \vdots \\
|\phi(s_m)|(f'(s_m))_1 & \!\dots\! & |\phi(s_m)|(f'(s_m))_n
\end{pmatrix} \ds_{j_{|J|}} \dots \ds_{j_1} .
\end{align*}
(Indeed, expanding the determinant on the LHS into a sum over permutations of the matrix entries, each term is a product of $|J|$ integrals which we combine into one $|J|$-fold integral.)  The integrand on the RHS above can be factored as
\begin{equation*}
\tx{minor}_{J,K} \left\{ \begin{pmatrix}
|\phi(s_1)| & & \\
& \!\ddots\! & \\
& & |\phi(s_m)|
\end{pmatrix} \begin{pmatrix}
(f'(s_1))_1 & \!\dots\! & (f'(s_1))_n \\
\vdots & \!\ddots\! & \vdots \\
(f'(s_m))_1 & \!\dots\! & (f'(s_m))_n
\end{pmatrix} \right\} .
\end{equation*}
The first matrix has positive determinant because $\phi$ is nonvanishing on each $I_j$ by construction.  Therefore, it suffices to show that $f'(s)$ is a strictly totally positive kernel, i.e.~for all $0<s_1<\dots<s_\ell$ and $k_1<\dots<k_\ell$ the matrix
\begin{equation}
\big( (f'(s_i))_{k_j} \big)_{1\leq i \leq \ell,\, 1\leq j\leq\ell}
= \begin{pmatrix}
(2k_1+1) s_1^{2k_1} & \!\dots\! & (2k_\ell+1) s_1^{2k_\ell} \\
\vdots & \!\ddots\! & \vdots \\
(2k_1+1) s_\ell^{2k_1} & \!\dots\! & (2k_\ell+1) s_\ell^{2k_\ell} \\
\end{pmatrix}
\label{eq:unique params 3}
\end{equation}
has positive determinant.  Note that when $k_1,\dots,k_\ell$ is an arithmetic progression, this matrix is essentially a Vandermonde matrix with rows and columns multiplied by constants.  We will follow the classical argument for Vandermonde matrices.

First, we claim that the determinant is nonzero.  Suppose towards a contradiction that the determinant vanishes.  Then the columns would be linearly dependent, and so there would exist $\lambda_1,\dots,\lambda_\ell \in\R$ so that
\begin{equation*}
\sum_{j=1}^\ell \lambda_j s_i^{2k_j} = 0 \quad\tx{for }i=1,\dots,\ell .
\end{equation*}
(We absorbed the coefficients $2k_j+1$ into $\lambda_j$.)  This means that the polynomial
\begin{equation}
P(x) := \sum_{j=1}^\ell \lambda_j x^{2k_j}
\label{eq:unique params 2}
\end{equation}
has $\ell$ positive roots $s_1,\dots,s_\ell$.  To obtain a contradiction, we will prove that nontrivial polynomials of the form~\eqref{eq:unique params 2} can have at most $\ell-1$ positive roots by induction on $\ell$.  The base case $\ell=1$ is immediate.  For the inductive step, note that if $P(x)$ has $\ell$ positive roots, then $x^{-2k_1} P(x)$ is a polynomial with the same $\ell$ positive roots.  By Rolle's theorem, the polynomial $(x^{-2k_1} P(x))'$ therefore has $\ell-1$ positive roots.  On the other hand, $(x^{-2k_1} P(x))'$ is a polynomial of the form~\eqref{eq:unique params 2} for $\ell-1$, and so this contradicts the inductive hypothesis.

Lastly, we show that the determinant is positive.  Now that we know the determinant is nonzero, its sign must be is independent of the choice of $0<s_1<\dots<s_\ell$ and $k_1<\dots<k_\ell$ (but may depend on $\ell$).  Therefore, we may pick $k_j = j$ for $j=1,\dots,\ell$ so that we essentially have a Vandermonde matrix with determinant
\begin{equation*}
\det\begin{pmatrix}
3 s_1^{2} & 5s_1^4 & \!\dots\! & (2\ell+1) s_1^{2\ell} \\
3 s_2^{2} & 5s_2^4 & \!\dots\! & (2\ell+1) s_2^{2\ell} \\
\vdots & \vdots & \!\ddots\! & \vdots \\
3 s_\ell^{2} & 5s_\ell^4 & \!\dots\! & (2\ell+1) s_\ell^{2\ell} \\
\end{pmatrix}
= 3\cdot 5\cdots (2\ell+1)\, s_1^2\cdots s_\ell^2\, \prod_{j<k} (s_k^2 - s_j^2) .
\end{equation*}
This is positive for any $\ell$ since $0<s_1<\dots<s_\ell$.
\end{proof}

For future reference (cf.~\cref{thm:min seq lem 5}), we note that the proof of \cref{thm:unique params} can allow for repeated $\beta$ parameters, as long as the total number of values is still at most $n$.
\begin{cor}
\label{thm:unique params 2}
Fix $n\geq 1$.  Given constraints $e_1,\dots,e_n$, there is at most one choice of $N\geq 1$ and $\beta_1\geq\cdots\geq\beta_N>0$ attaining at most $n$ distinct values that satisfies
\begin{equation*}
(-1)^{m+1} \tfrac{2^{2m+1}}{2m+1} \sum_{j=1}^N \beta_j^{2m+1} = e_m \quad\tx{for }m=1,\dots,n .
\end{equation*}
\end{cor}
\begin{proof}
We repeat the proof of \cref{thm:unique params}.   Construct the step function $\phi$ so that~\eqref{eq:unique params 1} holds.  It only remains to show that there are still at most $n$ intervals $I_j$ on which $\phi$ is nonvanishing and has constant sign.  If $\beta_m = \wt{\beta}_{\wt{m}}$ for some $m$ and $\wt{m}$, then these terms can be canceled from~\eqref{eq:unique params 4} while retaining equality.  Consequently, the only new possibility for $\phi$ is that there may be a run of $\beta_m$ parameters with the same value and the same sign $\eps_j$.  This increases the size of the jumps of $\phi$ but does not affect the number of sign changes, and so the claim follows.
\end{proof}

By \cref{thm:unique params}, the map $(\beta_1,\dots,\beta_n) \mapsto (E_1(Q_{\bs{\beta},\bs{c}}),\dots,E_n(Q_{\bs{\beta},\bs{c}}))$ from the half-open simplex
\begin{equation}
\{ (\beta_1,\dots,\beta_n) \in \R^n : \exists N\tx{ with }\beta_1>\cdots>\beta_N>0,\ \beta_{N+1} = \dots = \beta_n = 0 \}
\label{eq:simplex}
\end{equation}
into $\M^n_n$ has a well-defined inverse.  In fact, the inverse is also continuous:
\begin{lem}
\label{thm:homeo}
The function $(\beta_1,\dots,\beta_n) \mapsto (E_1(Q_{\bs{\beta},\bs{c}}),\dots,E_n(Q_{\bs{\beta},\bs{c}}))$ is a homeomorphism from the simplex~\eqref{eq:simplex} onto $\M^n_n$.
\end{lem}
\begin{proof}
Let
\begin{equation*}
\Phi(\beta_1,\dots,\beta_n) = \bigg( \tfrac{8}{3} \sum_{m=1}^n \beta_m^3 ,\ -\tfrac{32}{5} \sum_{m=1}^n \beta_m^5 ,\ \dots ,\ (-1)^{n-1}\tfrac{2^{2n+1}}{2n+1} \sum_{m=1}^n \beta_m^{2n+1} \bigg)
\end{equation*}
denote this function, which maps into $\M^n_n$ by definition of $\M^n_n$.  Each component of $\Phi$ is a polynomial, and so $\Phi$ is smooth.  By \cref{thm:unique params}, we also know that $\Phi$ is a bijection from the simplex~\eqref{eq:simplex} onto the set of constraints $\M^n_n$.  

It remains to show that $\Phi^{-1}$ is continuous.  Fix an open subset $V\subset\M^n_n$ and let $\bigcup_{m=1}^\infty K_m$ be a compact exhaustion of $\M^n_n$.  Recall the elementary topology fact that if $f:X\to Y$ is a continuous bijection between topological spaces with $X$ compact and $Y$ Hausdorff, then $f$ is a homeomorphism.  As the map $\Phi$ is also proper, then $\Phi$ is a homeomorphism from $\Phi^{-1}(K_m)$ to $K_m$ for all $m$.  Therefore $\Phi^{-1}(V\cap K_m)$ is relatively open in $\Phi^{-1}(K_m)$ for all $m$, and hence $\Phi^{-1}(V)$ is open.
\end{proof}

Given constraints $(e_1,\dots,e_n)\in\M^n_n$, let $\beta_1 > \cdots > \beta_N>0$ be the unique set of parameters with $N\leq n$ and $Q_{\bs{\beta},\bs{c}}$ satisfying the constraints.  We define
\begin{equation}
C(e_1,\dots,e_n) := E_{n+1}(Q_{\bs{\beta},\bs{c}})
\label{eq:min val}
\end{equation}
to be the value of the next functional for these parameters.  In proving \cref{thm:var char}, we will ultimately show that $C(e_1,\dots,e_n)$ is the minimum of $E_{n+1}$ subject to the constraints $e_1,\dots,e_n$.

In order to do this, we will first need some properties of $C$:
\begin{lem}
\label{thm:min val is dec}
The function $C:\M^n_n \to \R$ defined in~\eqref{eq:min val} is continuous and is decreasing in each variable.  Moreover, on the interior of $\M^n_n$, $C(e_1,\dots,e_n)$ is continuously differentiable and satisfies $\frac{\partial C}{\partial e_j} < 0$ for $j=1,\dots,n$.
\end{lem}
\begin{proof}
We write
\begin{equation}
C(e_1,\dots,e_n) = (-1)^n \tfrac{2^{2n+3}}{2n+3} \big( \beta_1^{2n+3} + \dots + \beta_n^{2n+3} \big) ,
\label{eq:min val is dec 2}
\end{equation}
where $(\beta_1,\dots,\beta_n)$ is the unique solution to
\begin{equation}
e_1 = \tfrac{8}{3} \sum_{m=1}^n \beta_m^3 , \quad
e_2 = - \tfrac{32}{5}  \sum_{m=1}^n \beta_m^5 , \quad \dots,\quad
e_n = (-1)^{n-1}\tfrac{2^{2n+1}}{2n+1}  \sum_{m=1}^n \beta_m^{2n+1}
\label{eq:min val is dec 1}
\end{equation}
in the simplex~\eqref{eq:simplex}, guaranteed by \cref{thm:unique params}.  Note that $C$ is continuous as the composition of the inverse of the homeomorphism in \cref{thm:homeo} with a polynomial.

Consequently, it suffices to show that $\frac{\partial C}{\partial e_j} < 0$ for $j=1,\dots,n$ on the interior of $\M^n_n$.  By \cref{thm:homeo}, the interior of $\M^n_n$ corresponds to the set of $n$ positive parameters $\beta_1 > \cdots> \beta_n>0$.  In other words, the interior of $\M^n_n$ is the set of constraints $(e_1,\dots,e_n)$ which correspond to $n$-solitons.

We will now compute $\frac{\partial C}{\partial e_j}$ assuming $\beta_1 > \cdots> \beta_n>0$.  Differentiating the constraints~\eqref{eq:min val is dec 1} with respect to $e_j$, we see that
\begin{equation}
\begin{pmatrix}
8\beta_1^2 & 8\beta_2^2 & \!\dots\! & 8\beta_n^2 \\[0.5em]
-32\beta_1^4 & -32\beta_2^4 & \!\dots\! & -32\beta_n^4 \\
\vdots & \vdots & \!\ddots\! & \vdots \\[0.5em]
(-1)^{n-1}2^{2n+1}\beta_1^{2n} & (-1)^{n-1}2^{2n+1}\beta_2^{2n} & \!\dots\! & (-1)^{n-1}2^{2n+1}\beta_n^{2n}
\end{pmatrix} \begin{pmatrix}
\frac{\partial\beta_1}{\partial e_j} \\[0.5em]
\frac{\partial\beta_2}{\partial e_j} \\[0.5em] 
\vdots \\[0.5em] 
\frac{\partial\beta_n}{\partial e_j}
\end{pmatrix}
\label{eq:min val is dec 3}
\end{equation}
is equal to the $j$th coordinate vector $(0,\dots,0,1,0,\dots,0)$.  This is a Vandermonde matrix after pulling out common factors from each row and column, and thus it has determinant
\begin{equation*}
8 \cdot (-32) \cdots (-1)^{n-1}2^{2n+1}\, \beta_1^2 \cdots \beta_n^2\, \prod_{j<k} (\beta_k^2 - \beta_j^2) .
\end{equation*}
This expression is nonvanishing since $\beta_1>\cdots>\beta_n>0$, and so we conclude that the partial derivatives $\frac{\partial \beta_k}{\partial e_j}$ exist and are uniquely determined by the matrix product~\eqref{eq:min val is dec 3}.  On the other hand, differentiating~\eqref{eq:min val is dec 2} with respect to $e_j$ yields
\begin{equation*}
\frac{\partial C}{\partial e_j}
= (-1)^n 2^{2n+3} 
\begin{pmatrix} \beta_1^{2n+2} & \!\dots\! & \beta_n^{2n+2} \end{pmatrix}
\begin{pmatrix}
\frac{\partial\beta_1}{\partial e_j} \\
\vdots \\
\frac{\partial\beta_n}{\partial e_j}
\end{pmatrix} .
\end{equation*}
We can determine the column vector on the RHS via~\eqref{eq:min val is dec 3}.  Collecting these equations for $j=1,\dots,n$, we obtain the matrix equation
\begin{equation}
\begin{aligned}
&\begin{pmatrix} \frac{\partial C}{\partial e_1} & \!\dots\! & \frac{\partial C}{\partial e_n} \end{pmatrix}
\begin{pmatrix}
8\beta_1^2 & \!\dots\! & 8\beta_n^2 \\
\vdots & \!\ddots\! & \vdots \\
(-1)^{n-1}2^{2n+1}\beta_1^{2n} & \!\dots\! & (-1)^{n-1}2^{2n+1}\beta_n^{2n}
\end{pmatrix} \\
&= (-1)^n 2^{2n+3}
\begin{pmatrix} \beta_1^{2n+2} & \!\dots\! & \beta_n^{2n+2} \end{pmatrix} ,
\end{aligned}
\label{eq:min val is dec 4}
\end{equation}
where we moved the matrix to the LHS to avoid inverting it.  We have already seen that this matrix is invertible, and so we conclude that the partial derivatives $\frac{\partial C}{\partial e_j}$ exist and are uniquely determined by the above equality.

In order to compute the derivatives $\frac{\partial C}{\partial e_j}$, we will harness the classical role of Vandermonde matrices in polynomial interpolation.  Specifically, reading off the $n$ components of the equality~\eqref{eq:min val is dec 4}, we see that the derivatives $\frac{\partial C}{\partial e_j}$ are the coefficients $C_j$ of the polynomial
\begin{equation*}
P(x) := 8C_1 x - 32 C_2 x^2 + \dots + (-1)^{n-1}2^{2n+1}C_nx^n - (-1)^n 2^{2n+3}x^{n+1}
\end{equation*}
which satisfies
\begin{equation*}
P(\beta_m^2) = 0\quad\tx{for }m=1,\dots,n .
\end{equation*}
As $\beta_1>\cdots>\beta_n>0$, there is only one such polynomial, namely,
\begin{equation*}
P(x) = (-1)^{n+1}2^{2n+3} x \prod_{m=1}^n (x-\beta_m^2) .
\end{equation*}
Therefore, the coefficients $C_j$ are given by Vieta's formulas:
\begin{equation}
\begin{aligned}
8C_1 &= (-1)^{n+1}2^{2n+3} (-1)^n \prod_{j=1}^n \beta_j^2 , \\[-0.5em]
-32C_2 &= (-1)^{n+1}2^{2n+3}(-1)^{n-1} \sum_{k=1}^n \prod_{j\neq k} \beta_j^2 , \\[-1em]
&\ \, \vdots \\[-1em]
(-1)^{n-1}2^{2n+1}C_n &= (-1)^{n+1}2^{2n+3}(-1)\sum_{j=1}^n \beta_j^2 ,
\end{aligned}
\label{eq:min val is dec 5}
\end{equation}
where the RHS for $C_j$ involves the $(n-j+1)$st elementary symmetric polynomial in $\beta_1^2,\dots,\beta_n^2$.  In particular, we see that each $\frac{\partial C}{\partial e_j} = C_j$ is given by $(-1)^{2n+1} = -1$ times a strictly positive quantity, and hence is strictly negative as desired.
\end{proof}

In order to employ that $C$ is decreasing, we will also need to know that the set $\M^n_n$ is downward closed within $\bigcup_{N\geq 0} \M^n_N$ in the following sense:
\begin{lem}
\label{thm:down closed}
If the constrains $\wt{e}_1,\dots,\wt{e}_n$ are in $\M^n_N$ for some $N$ and
\begin{equation*}
\wt{e}_1 \leq e_1 , \quad\dots,\quad \wt{e}_n\leq e_n
\end{equation*}
for some $(e_1,\dots,e_n) \in \M^n_n$, then $(\wt{e}_1,\dots,\wt{e}_n)\in\M^n_n$.
\end{lem}
\begin{proof}
Let $\wt{\beta}_1,\dots,\wt{\beta}_N>0$ be the $\beta$ parameters of the multisoliton which witnesses the constraints $\wt{e}_1,\dots,\wt{e}_n$.  The case $N\leq n$ is immediate, so assume that $N\geq n+1$.  Let
\begin{equation*}
\alpha_j = \sum_{m=1}^N \wt{\beta}_m^{j} , \quad j=3,5,\dots,2n+1
\end{equation*}
denote the odd moments of $\wt{\beta}_1,\dots,\wt{\beta}_N$.  Define
\begin{equation*}
\Gamma := \bigg\{ (x_1,\dots,x_N) \in\R^N : x_1,\dots,x_N\geq 0,\ \sum_{m=1}^N x_m^j = \alpha_j\tx{ for }j=3,5,\dots,2n+1 \bigg\}
\end{equation*}
to be the set of parameters in $\R^N$ satisfying the constraints, which is nonempty because it contains $(\wt{\beta}_1,\dots,\wt{\beta}_N)$.

It suffices to show that the intersection of $\Gamma$ with the $n$-dimensional boundary face $\{ (x_1,\dots,x_N) : x_{n+1},\dots,x_N = 0 \}$ is nonempty, since a point $(x_1,\dots,x_n,0,\dots,0)$ provides the $n$-soliton parameters that we seek.  The case $n=1$ is immediate as $\Gamma$ is just an $\ell^3$-sphere, and so we may assume that $\Gamma$ is the intersection of $n\geq 2$ constraints. 

To accomplish this, consider the set ``between'' the constraints
\begin{align*}
\Omega :=  \bigg\{ (x_1,\dots,x_N)\in\R^N :{} &x_1,\dots,x_N\geq 0, \\[-0.5em]
&(-1)^k \sum_{m=1}^N x_m^{2k+1} \leq (-1)^k \alpha_{2k+1}\tx{ for }k=1,\dots,n \bigg\} .
\end{align*}
Unlike $\Gamma$, we already know that the intersection of $\Omega$ with the boundary face $\{ (x_1,\dots,x_N) : x_{n+1},\dots,x_N = 0 \}$ is nonempty by premise.  Indeed, as $(e_1,\dots,e_n) \in \M^n_n$, then there exist $x_1,\dots,x_n$ so that
\begin{equation*}
(-1)^{k+1} \tfrac{2^{2k+1}}{2k+1} \sum_{m=1}^n \beta_m^{2k+1} = e_k \geq \wt{e}_k = (-1)^{k+1} \tfrac{2^{2k+1}}{2k+1} \alpha_{2k+1} \quad\tx{for }k=1,\dots,n .
\end{equation*}
(This premise is in fact necessary, as the sets $\Omega$ and $\Gamma$ may not intersect the boundary face $\{ (x_1,\dots,x_N) : x_{n+1},\dots,x_N = 0 \}$ in general; cf.~\cref{thm:N gas lem 1}.)  Note that $\Omega$ is also bounded, because each coordinate $x_j$ is bounded by $\alpha_5^{1/5}$ since $n\geq 2$.  As $\Omega \cap \{ (x_1,\dots,x_N) : x_{n+1},\dots,x_N = 0 \}$ is nonempty, closed, and bounded, there exists some point $(\beta_1,\dots,\beta_n,0,\dots,0)$ in this intersection that minimizes the $(n+1)$st odd moment $(-1)^n \sum x_m^{2n+3}$.  This point must actually lie in $\Gamma$, because \cref{thm:min val is dec} tells us that the value $(-1)^n \sum x_m^{2n+3}$ is an individually decreasing function of the moments on $\Omega \cap \{ (x_1,\dots,x_N) : x_{n+1},\dots,x_N = 0 \}$.
\end{proof}


\section{Global minimizers}
\label{sec:min}

We will prove \cref{thm:var char} over the course of this section by induction on $n$.  We begin with the base case $n=0$.  The conclusion is immediate since
\begin{equation*}
E_1(u) = \int_{-\infty}^\infty \tfrac{1}{2} u^2 \dx \geq 0\quad\tx{for all }u\in L^2(\R) ,
\end{equation*}
with equality if and only if $u$ is equal to the zero-soliton $q(x)\equiv 0$.

Next, we turn to the inductive step.  Suppose that $n\geq 1$ and that \cref{thm:var char} holds for $1,2,\dots,n-1$.  This inductive hypothesis yields the following fact for $\M^n_n$:
\begin{lem}
\label{thm:rel open}
The set $\M^n_n$ is a relatively open subset of the feasible constraints $\F^n$ (with respect to the topology on $\R^n$).
\end{lem}
\begin{proof}
Given $(e_1,\dots,e_n)$ in the interior of $\M^n_n$, we know that the numer of $\beta$ parameters is $n$ by \cref{thm:homeo}.  Therefore, for each $k\leq n-1$ we can increase and decrease $E_{k+1}$ while preserving $E_1,\dots,E_k$ by \cref{thm:beta mom wiggle} since $n$ is strictly larger than $k$.  Moreover, \cref{thm:homeo} implies that the set $\M^n_n \sm (\tx{int}\,\M^n_n)$ corresponds to multisolitons of degree at most $n-1$, and thus lie on the boundary of $\F^n$ by the inductive hypothesis that \cref{thm:var char} holds for each $k=1,\dots,n-1$.	
\end{proof}

Next, we will prove the first half of the inductive step: that multisolitons of degree at most $n$ are global constrained minimizers.
\begin{thm}
\label{thm:minimizers}
Given constraints $(e_1,\dots,e_n) \in \M^n_n$, there exists a unique integer $N\leq n$ and parameters $\beta_1>\cdots>\beta_N>0$ so that the multisoliton $Q_{\bs{\beta},\bs{c}}$ lies in $\con$ for some (and hence all) $\bs{c}\in\R^N$.  Moreover, we have
\begin{equation*}
E_{n+1}(u) \geq E_{n+1}(Q_{\bs{\beta},\bs{c}}) \quad \tx{for all }u\in\con.
\end{equation*}
\end{thm}
\begin{proof}
We first prove the inequality for $u$ Schwartz.  In this case, the trace formula~\eqref{eq:en trace} allows us to write
\begin{align*}
E_{n+1}(u)
&= \tfrac{2^{2n+2}}{\pi} \int_{-\infty}^\infty k^{2n+2} \log|a(k;u)|\dk + (-1)^n \tfrac{2^{2n+3}}{2n+3} \sum_{m=1}^N \beta_m^{2n+3} .
\intertext{The integrand is nonnegative by~\eqref{eq:trans coeff 1}, so we can omit the integral to obtain the inequality}
&\geq (-1)^n \tfrac{2^{2n+3}}{2n+3} \sum_{m=1}^N \beta_m^{2n+3} .
\intertext{Note that the $\beta$ parameters do not satisfy the constraints $e_j$, but rather the smaller constraints $e_j - \frac{2^{2j}}{\pi} \int k^{2j} \log|a|\dk$ because the moments of $\log|a(k;u)|$ may not vanish.  Nevertheless, as $\M^n_n$ is downward closed by \cref{thm:down closed}, we know that these constraints are still attainable by a multisoliton of degree at most $n$ and so \cref{thm:beta mom wiggle} implies that}
&\geq \textstyle C\big( e_1 - \tfrac{4}{\pi} \int k^2 \log|a|\dk ,\ \dots ,\ e_n - \tfrac{2^{2n}}{\pi} \int k^{2n} \log|a|\dk \big) .
\intertext{Finally, $C$ is individually decreasing in each variable by \cref{thm:min val is dec}, and so we conclude}
&\geq C( e_1, \dots , e_n )
\end{align*}
as desired.

For general $u\in H^n(\R)$, we approximate by a sequence of Schwartz functions.  The constraints $e_1,\dots,e_n$ and minimum value $C$ for the approximate functions will converge by the continuity of $E_1,\dots,E_n : H^n(\R) \to \R$ and $C:\M^n_n\to \R$, the latter of which we proved in \cref{thm:min val is dec}.  Moreover, the constraints $e_1,\dots,e_n$ for the approximate functions eventually lie in $\M^n_n$ by \cref{thm:rel open}.
\end{proof}

To conclude the inductive step of~\cref{thm:var char}, it remains to show that any other constrained minimizer must also be a multisoliton with the right $\beta$ parameters:
\begin{thm}
\label{thm:uniqueness}
Suppose we have constraints $(e_1,\dots,e_n) \in \M^n_n$ and that $q\in \con$ minimizes $E_{n+1}(u)$ over $\con$.  Then $q = Q_{\bs{\beta},\bs{c}}$ for some $\bs{c}\in\R^N$, where $\bs{\beta}\in\R^N$ are the unique parameters satisfying the constraints guaranteed by \cref{thm:minimizers}.
\end{thm}

We break the proof of \cref{thm:uniqueness} into steps, with the overarching assumption that $q\in H^n(\R)$ is a constrained minimizer of $E_{n+1}$.

In order to analyze $q$ using the trace formulas, we first need to know that $q$ is sufficiently regular so that we may construct $a(k;q)$:
\begin{lem}
\label{thm:schwartz}
If $q$ is a constrained minimizer in the sense of \cref{thm:uniqueness}, then $q$ is Schwartz.
\end{lem}

We will use that $q$ solves the Euler--Lagrange equation (cf.~\eqref{eq:schwartz 1}) to show that $q$ is both infinitely smooth and exponentially decaying.  As we will see shortly, smoothness follows from classical ODE theory because $q\in H^n(\R)$ \emph{a priori}.  On the other hand, exponential decay is more delicate: even though we know $q(x) \to 0$ as $x\to\pm\infty$ (since $q\in H^1(\R)$), there \emph{do} exist multipliers $\lambda_1,\dots,\lambda_n$ so that~\eqref{eq:schwartz 1} admits algebraically decaying solutions as $x\to\pm\infty$.  For example, $u(x) = 2x^{-2}$ is a (meromorphic) solution to~\eqref{eq:schwartz 1} for all $n\geq 1$ with multipliers $\lambda_1 = \dots = \lambda_n = 0$.  This is the beginning of an infinite family of solutions called the \emph{algebro-geometric solutions} to the stationary KdV hierarchy (see \cite{Gesztesy2003}*{\S1.3} for details).  Morally, this does not pose an obstruction here because the multipliers must be negative for a minimizer (cf.~\eqref{eq:schwartz 3}), as is the case for any multisoliton.

\begin{proof}
As a critical point of $E_{n+1}$, $q$ satisfies the Euler--Lagrange equation
\begin{equation}
\nabla E_{n+1}(q) = \lambda_1 \nabla E_1(q) + \lambda_2 \nabla E_2(q) + \dots + \lambda_n\nabla E_n(q)
\label{eq:schwartz 1}
\end{equation}
for some Lagrange multipliers $\lambda_1,\dots,\lambda_n\in\R$.  This assumes that the gradients $\nabla E_1(q)$, $\dots$, $\nabla E_n(q)$ are linearly independent; however, the other case is analogous, since a linear dependence can be written as an equation of the form~\eqref{eq:schwartz 1} for some smaller $n$.

First we show that $q$ is infinitely smooth.  As $q\in H^n(\R)$, $q$ only solves~\eqref{eq:schwartz 1} in the sense of distributions \ti{a priori}.  The highest order term in~\eqref{eq:schwartz 1} is $q^{(2n)}$, and it only appears in $\nabla E_{n+1}(q)$.  Isolating this term, we obtain
\begin{equation}
q^{(2n)} = P\big(q,q',\dots,q^{(2n-2)}\big)
\label{eq:schwartz 2}
\end{equation}
for a polynomial $P$.  Note that product terms $q^{(\gamma_1)}\cdots q^{(\gamma_d)}$ satisfy $\sum \gamma_j \leq 2n-2$ by the scaling requirement~\eqref{eq:en 2}.  In particular, if $q\in H^s$ with $s\geq n$, then RHS\eqref{eq:schwartz 2} is in $H^{-s+2}$.  (For example, $qq^{(2n-2)} \in H^{-s+2}$ because $q^{(2n-2)} \in H^{s-(2n-2)} \subset H^{-s+2}$ and $q\in H^s \subset H^{s-2}$.)  Beginning with $q\in H^n$, the equation~\eqref{eq:schwartz 2} tells us that $q^{(2n)}$ is in $H^{-n+2}$, and so we conclude that $q\in H^{n+2}$.  Now taking $q\in H^{n+2}$ as input, the equation~\eqref{eq:schwartz 2} then tells us that $q^{(2n)}$ is in $H^{-n+4}$, and so we conclude that $q\in H^{n+4}$.  Iterating, we conclude that $q$ is in $H^s$ for all $s>0$ and hence is smooth.

Next, we claim that $q$ decays exponentially as $x\to\pm\infty$.  Our salvation here is that because $q$ is a minimizer (and not merely a critical point), we have restrictions on the Lagrange multipliers.  If the constraints are in $\M^n_n \sm (\tx{int}\,\M^n_n)$, then $q$ is a minimizer of $E_m$ for some $m\leq n$ and thus $q$ is a multisoliton of degree at most $m$ by the inductive hypothesis that \cref{thm:var char} holds for $m-1$.  So assume that the constraints $(e_1,\dots,e_n)$ are in the interior of $\M^n_n$.  Consequently, we know from \cref{thm:min val is dec} that the minimum value $C(e_1,\dots,e_n)$ is a $C^1$ function in a neighborhood of $(e_1,\dots,e_n)$ and
\begin{equation}
\lambda_j = \frac{\partial C}{\partial e_j} (e_1,\dots,e_n) < 0 \quad\tx{for }j=1,\dots,n.
\label{eq:schwartz 3}
\end{equation}
The equality above is a general fact about Lagrange multipliers called the \emph{envelope theorem}, and has applications to economics.  (Cf.~\cite{Takayama1985}*{Th.~1.F.4} and the corollary in Ex.~2 for a proof.  As we know that all of the derivatives exist, this purely algebraic proof for the finite dimensional case still applies.)

From the quadratic terms of the energies~\eqref{eq:en}, we see that the linear part of the Euler--Lagrange equation~\eqref{eq:schwartz 1} is
\begin{equation*}
Lu := (-1)^n u^{(2n)} + (-1)^{n}\lambda_n u^{(2n-2)} + \dots +  \lambda_2 u''  - \lambda_1 u .
\end{equation*}
The constant coefficients of this operator are alternating in sign by~\eqref{eq:schwartz 3}, and consequently it has no purely imaginary eigenvalues.  Indeed, if $\xi$ is purely imaginary then all the terms in the polynomial
\begin{equation*}
(-1)^n \xi^{2n} + (-1)^{n}\lambda_n \xi^{2n-2} + \dots +  \lambda_2 \xi^2  - \lambda_1
\end{equation*}
are nonnegative, and so the polynomial is bounded below by $-\lambda_1 > 0$.  As the Euler--Lagrange equation~\eqref{eq:schwartz 1} is an ODE of order $2n$, we may view it as a first-order system of ODEs in the variables $(q,q',\dots,q^{(2n-1)}) \in \R^{2n}$.  We just showed that the origin in $\R^{2n}$ is a hyperbolic fixed point for this system, and so the stable manifold theorem \cite{Coddington1955}*{Ch.~13 Th.~4.1} tells us that there exists a corresponding stable manifold in a neighborhood of the origin.  We already know that $q^{(j)}(x) \to 0$ as $x\to\pm\infty$ for all $j\geq 0$ (since $q\in H^{j+1}$), and so eventually $(q,q',\dots,q^{(2n-1)})$ remains in a small neighborhood of the origin in $\R^{2n}$ for all $x$ sufficiently large.  By \cite{Coddington1955}*{Ch.~13 Th.~4.1}, this can only happen if $q(x)$ is on the stable manifold and hence decays exponentially as $x\to\pm\infty$.
\end{proof}

Now that we know $q\in\schwar(\R)$, we have the trace formulas~\eqref{eq:en trace} at our disposal.  Next, we show that the constrained minimizer $q$ must satisfy $|a(k;q)| \equiv 1$ on $\R$:
\begin{lem}
\label{thm:a mom wiggle}
Let $q\in\schwar(\R)$ such that $|a(k;q)| \neq 1$ for some $k\in\R$.  Then there exists some $\wt{q} \in \schwar(\R)$ with
\begin{equation*}
E_m(\wt{q}) = E_m(q) \quad \tx{for } m=1,\dots,n , \quad\tx{but}\quad
E_{n+1}(\wt{q}) < E_{n+1}(q) .
\end{equation*}
In particular, a constrained minimizer $q$ in the sense of \cref{thm:uniqueness} must satisfy $|a(k;q)| = 1$ for all $k\in\R$.
\end{lem}
\begin{proof}
We will only modify the transmission coefficient of $q$ and leave the bound states $-\beta_1^2,\dots,\beta_N^2$ unchanged.  Specifically, we will wiggle $\log|a(k;q)|$ via the implicit function theorem in a way that decreases its $(n+1)$st moment in the trace formulas~\eqref{eq:en trace} while keeping the first $n$ moments constant.  Then we will reconstruct the new potential $\wt{q}$ using inverse scattering theory.

Let $\psi_1,\dots,\psi_{n+1} \in C_c^\infty(\R)$ be even functions to be chosen later.  Define the function $f:\R^{n+1}\to\R^n$ by
\begin{equation*}
f(x_1,\dots,x_{n+1})
= \begin{pmatrix}
\int k^2 \big[ \log|a(k;q)| + x_1\psi_1(k) + \dots + x_{n+1}\psi_{n+1}(k) \big]\dk \\[0.3em]
\int k^4 \big[ \log|a(k;q)| + x_1\psi_1(k) + \dots + x_{n+1}\psi_{n+1}(k) \big]\dk \\
\vdots \\
\int k^{2n} \big[ \log|a(k;q)| + x_1\psi_1(k) + \dots + x_{n+1}\psi_{n+1}(k) \big]\dk
\end{pmatrix} .
\end{equation*}
This function has derivative matrix
\begin{equation*}
Df(0,\dots,0)
= \left( \begin{array}{ccc;{4pt/2pt}c}
\int k^2 \psi_1(k)\dk & \!\dots\! & \int k^2 \psi_n(k)\dk & \int k^2\psi_{n+1}(k)\dk \\
\int k^4 \psi_1(k)\dk & \!\dots\! & \int k^4 \psi_n(k)\dk & \int k^4\psi_{n+1}(k)\dk \\
\vdots & \!\ddots\! & \vdots & \vdots \\
\int k^{2n} \psi_1(k)\dk & \!\dots\! & \int k^{2n} \psi_n(k)\dk & \int k^{2n}\psi_{n+1}(k)\dk
\end{array} \right)
\end{equation*}
at the origin.  If we replace each $\psi_j$ by the even extension $\tfrac{1}{2}(d\del_{- k_j} + d\del_{k_j})$ of a Dirac delta mass at $k_j>0$, then the left $n\times n$ block matrix becomes the Vandermonde matrix
\begin{equation*}
\begin{pmatrix}
k_1^2 & k_2^2 & \!\dots\! & k_n^2 \\	
k_1^4 & k_2^4 & \!\dots\! & k_n^4 \\
\vdots & \vdots & \!\ddots\! & \vdots \\
k_1^{2n} & k_2^{2n} & \!\dots\! & k_n^{2n}
\end{pmatrix} 
\quad\tx{with determinant}\quad 
k_1^2 \cdots k_n^2\, \prod_{i<j} (k_j^2 - k_i^2) .
\end{equation*}
This determinant is nonvanishing provided that we pick the $k_i$ positive and distinct.

As $a(k;q)$ is a continuous and even function of $k\in\R$ by the reality condition~\eqref{eq:trans coeff 3}, we may pick $n+1$ distinct points $k_1,\dots,k_{n+1}$ in $\{ k > 0 : |a(k;q)| \neq 1 \}$.  We will take $\psi_1,\dots,\psi_{n}$ to be mollifications of $\tfrac{1}{2}( d\del_{- k_j} + d\del_{k_j})$ for $j=1,\dots,n$ by a smooth and even function.  If we take the mollifier to have sufficiently small support, then $\psi_1,\dots,\psi_{n}$ will have disjoint supports within $\{ k \neq 0 : |a(k;q)| \neq 1 \}$.  Taking the support of the mollifier to be even smaller if necessary, the above computation shows that the left $n\times n$ block of $Df(0,\dots,0)$ is invertible.  Now the implicit function theorem implies that there exists $\eps>0$ and $C^1$ functions $x_1(x_{n+1})$, $\dots$, $x_n(x_{n+1})$ defined on $(-\eps,\eps)$ so that
\begin{equation}
f(x_1(x_{n+1}),\dots,x_n(x_{n+1}),x_{n+1})
\equiv f(0,\dots,0)
= \begin{pmatrix}
\int k^2 \log|a(k;q)|\dk \\
\int k^4 \log|a(k;q)|\dk \\
\vdots \\
\int k^{2n} \log|a(k;q)|\dk
\end{pmatrix} 
\label{eq:a mom wiggle 1}
\end{equation}
for $x_{n+1}\in (-\eps,\eps)$.

It remains to show that we can pick $x_{n+1}$ in a way that decreases the next $\log|a|$ moment.  To this end, we compute the derivative
\begin{align*}
&\frac{d}{dx_{n+1}} \bigg|_{x_{n+1}=0} \int k^{2n+2} \big[ \log|a| + x_1(x_{n+1})\psi_1 + \dots + x_n(x_{n+1})\psi_n + x_{n+1}\psi_{n+1} \big]\dk \\
&= \int k^{2n+2} \left[ 
\begin{pmatrix} \psi_1 & \!\dots\! & \psi_n \end{pmatrix} 
\begin{pmatrix} x'_1(0) \\ \vdots \\ x'_n(0) \end{pmatrix}
+\psi_{n+1} \right] \dk .
\end{align*}
The derivative of $x_1(x_{n+1})$, $\dots$, $x_n(x_{n+1})$ is determined by differentiating~\eqref{eq:a mom wiggle 1} at $x_{n+1} = 0$.  This yields
\begin{equation*}
\begin{pmatrix} x'_1(0) \\ \vdots \\ x'_n(0) \end{pmatrix}
= - \begin{pmatrix} \int k^2 \psi_1 & \!\dots\! & \int k^2\psi_n \\ \vdots & \!\ddots\! & \vdots \\ \int k^{2n}\psi_1 & \!\dots\! & \int k^{2n}\psi_n \end{pmatrix}^{\!-1}\!\!
\begin{pmatrix} \int k^2\psi_{n+1} \\ \vdots \\ \int k^{2n} \psi_{n+1} \end{pmatrix} .
\end{equation*}
Recall that the matrix above is invertible by our choice of $\psi_1,\dots,\psi_n$.  Inserting this into the derivative of the $(2n+2)$nd moment, the resulting matrix product is supported on the union of the supports of $\psi_1,\dots,\psi_n$ which is disjoint from the support of the other term $\psi_{n+1}$ in the integrand.  Therefore, we may pick another smooth and even function $\psi_{n+1}$ supported in a sufficiently small neighborhood of $\pm k_{n+1}$ so that the whole integral is nonzero.  It then follows that there exists $x_{n+1} \in (-\eps,\eps)$ sufficiently small so that
\begin{equation}
\log |a(k;q)| + x_1(x_{n+1})\psi_1 + \dots + x_n(x_{n+1}) \psi_n + x_{n+1}\psi_{n+1}
\label{eq:a mom wiggle 2}
\end{equation}
is nonnegative for $k\in\R$, decreases $E_{n+1}$, and preserves $E_1,\dots,E_n$.

It only remains to show that the density~\eqref{eq:a mom wiggle 2} corresponds to $\log|a(k;\wt{q})|$ for some $\wt{q}\in\schwar (\R)$.  To accomplish this, we will reconstruct $\wt{q}$ from its scattering data by verifying properties (i)-(vii) of \cref{thm:inv scat thy}. First, we require that the transmission coefficient satisfies
\begin{equation*}
\tfrac{1}{|T(k;\wt{q})|} = |a(k;\wt{q})| = \exp\big\{ \log |a(k;q)| + x_1(x_{n+1})\psi_1 + \dots + x_n(x_{n+1}) \psi_n + x_{n+1}\psi_{n+1} \big\}
\end{equation*}
for $k\in\R$.  We have $|T(k;\wt{q})|\leq 1$ because~\eqref{eq:a mom wiggle 2} is nonnegative.  As $a(k,q)$ extended to a bounded holomorphic to all of $\C^+$, it is in the Hardy space $\ms{H}^\infty(\C^+)$.  By~\cite{Garnett2007}*{Ch.~II Th.~4.4}, given $g \in L^\infty(\R)$ nonnegative, there exists a holomorphic function $h\in\ms{H}^\infty(\C^+)$ with $|h(k)|= g(k)$ for almost every $k\in\R$ if and only if
\begin{equation*}
\int_{-\infty}^\infty \frac{\log|g(k)|}{1+k^2}\dk > - \infty .
\end{equation*}
In particular, this property was satisfied by $g(k) = |a(k;q)|$.  As we have smoothly modified $\log |a(k;q)|$ on a compact subset, this condition is also satisfied for $g(k) = |a(k;\wt{q})|$ and so there must also exist a holomorphic extension $a(k;\wt{q}) := h(k)$ to $\C^+$.  This ensures that $T(k;\wt{q}) = \frac{1}{a(k;\wt{q})}$ satisfies the analyticity condition~(iii).

Next, we set $T_1(k;\wt{q}) = T_2(k;\wt{q}) = T(k;\wt{q})$ in accordance with the symmetry condition~(i).  We then require the modulus of the reflection coefficients satisfy $|R_1(k;\wt{q})| = |R_2(k;\wt{q})| = \sqrt{1-|T(k;\wt{q})|^2}$ and the phases satisfy
\begin{equation*}
\frac{\Arg R_1(k;\wt{q}) + \Arg R_2(k;\wt{q})}{2} = \tfrac{\pi}{2} - \Arg T(k;\wt{q}) 
\end{equation*}
for $k>0$ to ensure that the unitary condition~(ii) holds.  We also need to construct $R_1,R_2$ so that condition~(v) on the rate at $k=0$ still holds.  We are still free to specify the difference $\Arg R_1 - \Arg R_2$, which if $T(0;q) = 0$ we take to satisfy $\Arg R_1 \to \pi$ and $\Arg R_2 \to \pi$ as $k\downarrow 0$.  As we have only modified $T(k;q)$ on a compact subset of $\R\sm\{0\}$, altogether we conclude that the condition~(v) is satisfied.  We then extend $R_1$ and $R_2$ to $k<0$ according to the reality condition~(vi).

Note that the Fourier decay condition~(vii) is automatically satisfied because we have perturbed $R_1(k;q)$ and $R_2(k;q)$ smoothly.  Likewise, the asymptotics condition~(iv) is also satisfied because the coefficients have only been modified on a compact subset of $\R\sm\{0\}$.  The resulting potential $\wt{q}$ is also automatically Schwartz.  Indeed, inverse scattering theory reconstructs $\wt{q}$ via an explicit integral~\cite{Deift1979}*{\S4 Eq.~(1)$_R$} in terms of $R_1$ and the Jost function $f_1(x;k)$, which have only been smoothly modified on a compact subset of $\R\sm\{0\}$.

Lastly, we use \cref{thm:inv scat thy 2} to add the bound states $\beta_1,\dots,\beta_N$ of $q$ back to $\wt{q}$.  The formula~\eqref{eq:blaschke 2} for the new transmission coefficient shows that $|a(k;\wt{q})|$ is unchanged for $k\in\R$.  This, together with the construction~\eqref{eq:a mom wiggle 2} of $\log|a(k;\wt{q})|$, shows that $\wt{q}$ decreases $E_{n+1}$ while preserving $E_1,\dots,E_n$ as desired.
\end{proof}

Next, we will show that the requirement $|a(k;q)| \equiv 1$ on $\R$ forces $a(k;q)$ to be the finite Blaschke product~\eqref{eq:blaschke} on $\C^+$:
\begin{lem}
\label{thm:inner outer fac}
If $q$ is a constrained minimizer in the sense of \cref{thm:uniqueness}, then $q$ is a multisoliton.
\end{lem}
\begin{proof}
Let $\C^+ = \{ z\in\C : \Im z > 0 \}$ denote the upper half-plane and $\D = \{ z\in\C : |z| < 1\}$ denote the unit disk.  By \cref{thm:a mom wiggle}, we know that $|a(k;q)| \equiv 1$ on $\R$.  We also know by the asymptotics~\eqref{eq:trans coeff 2} that $a(k;q)$ tends to 1 as $k\to \infty$ within $\C^+$.  Applying the maximum modulus principle to the half-disks $\ol{\C^+} \cap \{ z : |z|\leq R\}$ and taking $R\to\infty$, we conclude that $k\mapsto a(k;q)$ maps $\C^+$ into $\ol{\D}$.

In particular, $a(\,\cdot\,;q)$ is in the Hardy space $\ms{H}^\infty(\C^+)$.  We may therefore apply inner-outer factorization~\cite{Garnett2007}*{Ch.~II Cor.~5.7} to obtain
\begin{equation*}
a(k;q) = e^{i\theta} B(k) S(k) F(k) ,
\end{equation*}
where $\theta\in\R$ is a constant, $B(k)$ is a Blaschke product, $S(k)$ is a singular function, and $F(k)$ is an outer factor.

First, we claim that $F(k)$ is constant.  Note that on $\R$ we have $|B(k)| \equiv 1$ everywhere and $|S(k)| \equiv 1$ almost everywhere.  Therefore $|F(k)| \equiv 1$ almost everywhere on $\R$.  As $F$ is an outer factor, $\log|F|$ in $\C^+$ is given by the Poisson integral over its boundary values.  As $\log|F| \equiv 0$ almost everywhere on $\R$, we conclude that $\log|F| \equiv 0$ on $\C^+$ and hence $F$ is constant.

Next, we claim that $S(k)$ is also constant.  Recall that if $S(k)$ is a singular function and $|S(k)|$ is continuous from $\C^+$ to any $k\in\R\cup\{\infty\}$, then $k$ is not in the support of the singular measure that defines $S$.  In our case we know that $a(k;q)$ extends continuously to $\R$ and $\infty$, and so we conclude that the measure for $S(k)$ vanishes identically.

Altogether we now have $a(k;q) \equiv e^{i\phi}B(k)$ for some constant $\phi\in\R$.  Taking $k\to +i\infty$ we have $a(k;q) \to 1$ by~\eqref{eq:trans coeff 2} and $B(k)\to 1$, and so we conclude that $a(k;q) \equiv B(k)$.  By \cref{thm:scat thy} we know that the zeros of $a(k;q)$ are purely imaginary and simple.  Therefore $a(k;q)$ takes the form~\eqref{eq:blaschke}, and so we conclude that $q$ is a multisoliton.
\end{proof}

Finally, to conclude the proof of \cref{thm:uniqueness}, we note that the degree of the multisoliton $q$ is at most $n$.  Otherwise, \cref{thm:beta mom wiggle} would imply that we could replace $q$ by another multisoliton that decreases $E_{n+1}$ while preserving $E_1,\dots,E_n$, which would contradict that $q$ is a minimizer.


\section{Orbital stability}
\label{sec:orb stab}

The goal of this section is to prove \cref{thm:orb stab}.  It will follow easily from the following property of minimizing sequences:
\begin{thm}
\label{thm:min seq}
Fix an integer $n\geq 1$.  If $(e_1,\dots,e_n) \in \M^n_n$ and $\{q_k\}_{k\geq 1} \subset H^n(\R)$ is a minimizing sequence:
\begin{equation}
E_1(q_k) \to e_1, \quad\dots, \quad E_n(q_k)\to e_n, \quad E_{n+1}(q_k) \to C(e_1,\dots,e_n)
\label{eq:min seq 1}
\end{equation}
as $k\to\infty$, then there exists a subsequence which converges in $H^n(\R)$ to the manifold of minimizing solitons $\{ Q_{\bs{\beta},\bs{c}} : \bs{c} \in \R^N\}$.
\end{thm}

We begin the proof of \cref{thm:min seq} by fixing a minimizing sequence $\{q_k\}_{k\geq 1}\subset H^n(\R)$ satisfying~\eqref{eq:min seq 1}.  The estimates~\eqref{eq:ests for en cts} that prove that $E_1,\dots,E_{n+1}$ are continuous functionals on $H^n(\R)$ show that the sequence $\{ q_k \}_{k\geq 1}$ is bounded in $H^n(\R)$.

Now that we know $\{ q_k \}_{k\geq 1}$ is bounded in $H^n(\R)$, we are able to apply a concentration compactness principle adapted to our variational problem.  Specifically, we will use the following statement associated to the embedding $H^n(\R) \hookrightarrow W^{n-1,3}(\R)$, whose formulation is inspired by~\cite{Killip2013}.  The choice of concentration compactness principle is not unique (cf.~\cite{Albert2021}*{\S3}), but we will see below that this choice turns out to be efficient (see, for example, the proof of \cref{thm:min seq lem 1}).
\begin{thm}
\label{thm:conc comp}
Fix an integer $n\geq 1$.  If $\{q_k\}_{k\geq 1}$ is a bounded sequence in $H^n(\R)$, then there exist $J^* \in \{0,1,\dots,\infty\}$, $J^*$-many profiles $\{\phi^j\}_{j=1}^{J^*} \subset H^n(\R)$, and $J^*$-many sequences $\{ x^j_k \}_{j=1}^{J^*} \subset \R$ so that along a subsequence we have the decomposition
\begin{equation}
q_k(x) = \sum_{j=1}^J \phi^j(x-x^j_k) + r^J_k(x) \quad\tx{for all } J\in\{0,\dots,J^*\} \tx{ finite}
\label{eq:conc comp decomp}
\end{equation}
that satisfies:
\begin{gather}
\lim_{J\to J^*}\, \limsup_{k\to\infty}\, \norm{r^J_k}_{W^{n-1,3}} = 0 ,
\label{eq:conc comp 1} \\
\lim_{k\to\infty}\, \bigg| \norm{q_k}_{H^n}^2 - \bigg( \sum_{j=1}^J \norm{\phi^j}_{H^n}^2 + \norm{r^J_k}_{H^n}^2 \bigg) \bigg| = 0 \quad\tx{for all }J\tx{ finite},
\label{eq:conc comp 2} \\
\lim_{J\to J^*}\, \bigg| \limsup_{k\to\infty}\, \norm{q_k}_{W^{n-1,3}}^3 - \sum_{j=1}^J \norm{\phi^j}_{W^{n-1,3}}^3 \bigg| = 0 ,
\label{eq:conc comp 3} \\
|x^j_k - x^{\ell}_k| \to \infty \quad\tx{as }  k\to\infty \tx{ whenever } j\neq \ell .
\label{eq:conc comp 4}
\end{gather}
\end{thm}

The $n=1$ case of \cref{thm:conc comp} is well-known~\cite{Hmidi2005}*{Prop.~3.1}; for a textbook presentation of such concentration compactness principles, we recommend~\cite{Killip2013}.  While it does not appear that \cref{thm:conc comp} for $n\geq 2$ has been recorded in the literature, it can be proved by exactly the same method (e.g.~\cite{Killip2013}*{Th.~4.7}) and we omit the details.

We apply this concentration compactness principle to the minimizing sequence $\{q_k\}_{k\geq 1}$ in \cref{thm:min seq}.  After passing to a subsequence, \cref{thm:conc comp} provides a number $J^*\in\{0,1,\dots,\infty\}$, $J^*$-many profiles $\{\phi^j\}_{j=1}^{J^*} \subset H^n(\R)$, and $J^*$-many sequences $\{ x^j_k \}_{j=1}^{J^*} \subset \R$ so that along a subsequence we have the decomposition~\eqref{eq:conc comp decomp} satisfying the properties~\eqref{eq:conc comp 1}-\eqref{eq:conc comp 4}.  We will ultimately show that each profile $\phi^j$ is a constrained minimizer of $E_{n+1}$, and hence is a multisoliton.

First, we will treat the case $J^* = 0$:
\begin{lem}
\label{thm:min seq lem 6}
If $J^* = 0$, then $e_1= \dots = e_n = 0$ and $q_k\to 0$ in $H^n(\R)$ as $k\to\infty$.
\end{lem}
\begin{proof}
The decomposition~\eqref{eq:conc comp decomp} reads $q_k = r^0_k$, and so
\begin{equation*}
E_2(q_k)
= E_2(r^0_k)
= \int \big\{ \tfrac{1}{2} \big[ (r^0_k)' \big]^2 + ( r^0_k)^3 \big\} \dx .
\end{equation*}
The second term in the integrand contributes $\snorm{ r^0_k }_{L^3}^3$, which we know vanishes in the limit $k\to\infty$ by~\eqref{eq:conc comp 1}.   The remaining term is nonnegative, and so we obtain
\begin{equation*}
\lim_{k\to\infty} E_2(q_k)
\geq 0 .
\end{equation*}
On the other hand, we know that this limit is attainable by a multisoliton since $(e_1,\dots,e_n) \in \M^n_n$, and so there exists $N\leq n$ and $\beta_1 > \cdots > \beta_N> 0$ so that
\begin{equation*}
\lim_{k\to\infty} E_2(q_k)
= -\tfrac{32}{5} \sum_{m=1}^N \beta_m^5
\leq 0 .
\end{equation*}
Together, we see that $E_2(q_k)\to 0$ and $N=0$.  The only multisoliton that can attain this value is the zero-soliton $q(x) \equiv 0$, and so we conclude that $e_1=\dots=e_n=0$.

Now we have
\begin{equation*}
0
= \lim_{k\to\infty} E_1(q_k)
= \lim_{k\to\infty} \tfrac{1}{2} \norm{q_k}_{L^2}^2 ,
\end{equation*}
and so $q_k \to 0$ in $L^2(\R)$.  Using the estimates~\eqref{eq:ests for en cts} that prove that $E_1,\dots,E_{n+1}$ are continuous functionals on $H^n(\R)$, we obtain $q_k\to 0$ in $H^n(\R)$ as desired.
\end{proof}

\Cref{thm:min seq lem 6} proves \cref{thm:min seq} in the case $J^*=0$, and so for the remainder of the section we assume $J^* \geq 1$.

Next, we show that our decomposition accounts for the entirety of the limiting value of each $E_{m}(q_k)$:
\begin{lem}
\label{thm:min seq lem 1}
For each $m=1,\dots,n+1$ we have
\begin{equation*}
 \lim_{J\to J^*}\, \limsup_{k\to\infty}\, \bigg| E_m(q_k) - \bigg[ \sum_{j=1}^J E_m(\phi^j) + E_m(r^J_k) \bigg] \bigg| = 0 .
\end{equation*}
\end{lem}
\begin{proof}
Fix $m$, and insert the decomposition~\eqref{eq:conc comp decomp} for $q_k$ into the expression~\eqref{eq:en} for $E_m$.  By the $H^n$-norm property~\eqref{eq:conc comp 2}, we see the quadratic terms of each $E_m$ cancel, leaving only cubic and higher order terms:
\begin{align*}
&\limsup_{k\to\infty}\, \bigg| E_m(q_k) - \bigg[ \sum_{j=1}^J E_m(\phi^j) + E_m(r^J_k) \bigg] \bigg| \\
&= \limsup_{k\to\infty}\, \bigg| \int \bigg\{ P_m\bigg( \sum_{j=1}^J \phi^j(x-x^j_k) + r^J_k \bigg) - \bigg[ \sum_{j=1}^J P_m(\phi^j) + P_m(r^J_k) \bigg] \bigg\} \dx \bigg| .
\end{align*}

Consider an arbitrary term $c_{\alpha_1,\dots,\alpha_d} u^{(\alpha_1)}\cdots u^{(\alpha_d)}$ of $P_m(u)$ for $u = q_k$, $\phi^j$, or $r^J_k$.  Expanding all products in the case $u = q_k = \sum \phi^j(x-x^j_k) + r^J_k$, we are left with a term of the form
\begin{equation*}
u_1^{(\alpha_1)}\cdots u_d^{(\alpha_d)}
\end{equation*}
where each $u_\ell$ for $\ell = 1,\dots,d$ is one of the profiles $\phi^j$, its translation $\phi^j(x-x^j_k)$, or the remainder $r^J_k$.

We claim that all of the terms with $u_1,\dots,u_d \neq r^J_k$ cancel; in other words,
\begin{equation}
\lim_{k\to\infty}\, \bigg| \int \bigg\{ P_m\bigg( \sum_{j=1}^J \phi^j(x-x^j_k) \bigg) - \sum_{j=1}^J P_m(\phi^j) \bigg\} \dx \bigg| = 0
\label{eq:orb stab 4}
\end{equation}
for all $J\leq J^*$ finite.  When all of the $u_\ell$ are given by the same translated profile $\phi^j(x-x^j_k)$, we can change variables $y=x-x^j_k$ in the integral and recover the corresponding term where $u_1,\dots,u_d = \phi^j(x)$.  When there are at least two different translated profiles, the integral vanishes in the limit $k\to\infty$ by the well-separation condition~\eqref{eq:conc comp 4} and approximating each $\phi^j$ by compactly-supported functions.

It remains to show that the remaining terms (which contain at least one factor of $r^J_k$) vanish in $L^1$.  Note that by the scaling requirement~\eqref{eq:en 2}, each order $\alpha_\ell$ is at most $m-2 \leq n-1$.  We estimate the highest order factor of $r^J_k$ in $L^3$, which is vanishing in the limit $k\to\infty$ and $J\to J^*$ by the small-remainder condition~\eqref{eq:conc comp 1}.  We then estimate the two other highest order factors $\phi^j$ or $r^J_k$ in $L^3$, and the remaining terms are bounded in $L^\infty$ since $\phi^j$ and $r^J_k$ are uniformly bounded in $H^n \hookrightarrow W^{n-1,\infty}$.  
\end{proof}

Next, we show that the quadratic term of $E_{n+1}(r^J_k)$ dominates as $k\to\infty$:
\begin{lem}
\label{thm:min seq lem 2}
For each $m=1,\dots,n+1$ we have
\begin{equation*}
\lim_{J\to J^*}\, \limsup_{k\to\infty}\, \Big| E_m(r^J_k) - \tfrac{1}{2} \bnorm{(r^J_k)^{(m-1)}}_{L^2}^2 \Big| = 0 .
\end{equation*}
\end{lem}
\begin{proof}
Fix $m$.  Using the expression~\eqref{eq:en} for $E_m$, we write
\begin{equation*}
E_m(r^J_k) - \tfrac{1}{2} \bnorm{(r^J_k)^{(m-1)}}_{L^2}^2
= \int P_m\big(r^J_k\big)\dx .
\end{equation*}
Consider the contribution of an arbitrary term $c_{\alpha_1,\dots,\alpha_d} u^{(\alpha_1)}\cdots u^{(\alpha_d)}$ of $P_m(u)$, where each order $\alpha_\ell$ is at most $m-2 \leq n-1$ by the scaling requirement~\eqref{eq:en 2}.  We estimate the three highest order factors of $r^J_k$ in $L^3$, which vanish in the limit $k\to\infty$ and $J\to J^*$ by the small-remainder condition~\eqref{eq:conc comp 1}.  We then estimate the remaining terms in $L^\infty$, which are all bounded since the sequence $r^J_k$ is uniformly bounded in $H^n \hookrightarrow W^{n-1,\infty}$.  Altogether, we conclude that every term vanishes in the limit $k\to\infty$ and $J\to J^*$.
\end{proof}

Next, we show that each profile $\phi^j$ is a constrained minimizer:
\begin{lem}
\label{thm:min seq lem 4}
For each $1\leq j \leq J^*$ finite, the profile $\phi^j$ minimizes $E_{n+1}(u)$ over all $u\in\con$ with the constraints $E_1(\phi^j)$, $\dots$, $E_n(\phi^j)$, and hence is a multisoliton $Q_{\bs{\beta}^j,\bs{c}^j}$ of degree at most $n$.
\end{lem}
\begin{proof}
Suppose towards a contradiction that there exists $j$ for which $\phi^j$ does not minimize $E_{n+1}$.  Then we can replace $\phi^j$ by another profile $\wt{\phi}^j \in H^n(\R)$ with
\begin{equation*}
E_m(\wt{\phi}^j) = E_m(\phi^j) \quad\tx{for }m=1,\dots,n , \quad\tx{but}\quad
E_{n+1}(\wt{\phi}^j) < E_{n+1}(\phi^j) .
\end{equation*}
Construct a new sequence $\{\widetilde{q}_k\}_{k\geq 1}$ given by the decomposition~\eqref{eq:conc comp decomp}, but with $\wt{\phi}^j$ in place of $\phi^j$.  This new sequence still satisfies the properties~\eqref{eq:conc comp 1}--\eqref{eq:conc comp 4}, and so by \cref{thm:min seq lem 1} we have
\begin{equation}
\label{eq:min seq 2}
\lim_{k\to\infty} E_m(\widetilde{q}_k) = e_m \quad\tx{for }m=1,\dots,n , \quad
\lim_{k\to\infty} E_{n+1}(\widetilde{q}_k) < \lim_{k\to\infty} E_{n+1}(q_k) .
\end{equation}
However, this contradicts that $\{q_k\}_{k\geq 1}$ was a minimizing sequence.  Therefore, we conclude that each $\phi^j$ is a constrained minimizer of $E_{n+1}$.

Applying our variational characterization (\cref{thm:var char}), we conclude that $\phi^j$ is a multisoliton of degree at most $n$.
\end{proof}

We now know that our profiles $\{ \phi^j \}_{j=1}^{J^*}$ are a (possibly infinite) collection of multisolitons $\{ Q_{\bs{\beta}^j,\bs{c}^j} \}_{j=1}^{J^*}$.  Concatenate all of the vectors $\bs{\beta}^j$ to form one (possibly infinite) string $\coprod_{j=1}^{J^*} \bs{\beta}^j$ of positive numbers which may contain repeated values.  Next, we show that all of the parameters together minimize $E_{n+1}$ subject to the constraints $e_1,\dots,e_n$:
\begin{lem}
\label{thm:min seq lem 5}
The concatenation $\coprod_{j=1}^{J^*} \bs{\beta}^j$ is equal to the unique set of parameters $\beta_1>\cdots>\beta_N>0$ satisfying the constraints $e_1,\dots,e_n$.  In particular, $J^*$ is finite.
\end{lem}
\begin{proof}
Consider the relaxed variational problem where we minimize $E_{n+1}(u)$ over the larger set
\begin{equation}
\{ u\in H^n(\R) : E_1(u) \leq e_1,\ \dots,\ E_n(u) \leq e_n \} .
\label{eq:con 2}
\end{equation}
As the minimum value $C(e_1,\dots,e_n)$ is strictly decreasing in each constraint by \cref{thm:min val is dec} and the set $\M^n_n$ of constraints is downward closed by \cref{thm:down closed}, then this relaxed minimization problem enjoys the same conclusions of \cref{thm:var char}.  We will ultimately show that the profiles $\{ Q_{\bs{\beta}^j,\bs{c}^j} \}_{1\leq j< J^*}$ together form a minimizer for this relaxed problem.

In the proof of \cref{thm:var char} we only needed to treat the case of finitely many $\beta$ parameters, but this is easily resolved as follows.  Suppose towards a contradiction that $J^* = \infty$.  We know that the third moment $\sum \beta_m^3$ of $\bs{\beta} = \coprod_{j=1}^{J^*} \bs{\beta}^j$ is finite by the trace formula~\eqref{eq:en} for $E_1$, and so we have $\beta_m\to 0$ as $m\to\infty$.  In particular, even though there may be repeated values in $\bs{\beta}$, there are at least $n+1$ distinct values of $\beta_m > 0$.  By \cref{thm:beta mom wiggle}, we may replace the first $n+1$ distinct values of $\beta_m$ to obtain new parameters $\wt{\bs{\beta}}^j$ so that $E_1,\dots,E_n$ are preserved, $E_{n+1}$ is decreased, and each $\bs{\beta}^j$ is still a set of multisoliton parameters.  Constructing a new sequence $\{\widetilde{q}_k\}_{k\geq 1}$ given by the decomposition~\eqref{eq:conc comp decomp} for the profiles $\wt{\phi}^j = Q_{\wt{\bs{\beta}}^j\bs{c}^j}$, we obtain a strictly better choice of minimizing sequence in the sense of~\eqref{eq:min seq 2}.  This contradicts that $\{q_k\}_{k\geq 1}$ was a minimizing sequence.

Now that we know $J^*<\infty$, \cref{thm:min seq lem 1} implies that
\begin{equation}
\sum_{j=1}^{J^*} E_m( Q_{\bs{\beta}^j,\bs{c}^j} )
\leq \limsup_{k\to\infty} \big[ E_m(q_k) - E_m(r^{J^*}_k) \big]
\label{eq:orb stab 2}
\end{equation}
for each $m=1,\dots,n+1$.  For $m\leq n$, we have $E_m(q_k) \to e_m$ by construction and
\begin{equation*}
\liminf_{k\to\infty} E_m(r^{J^*}_k) \geq 0
\end{equation*}
by \cref{thm:min seq lem 2}, and so RHS\eqref{eq:orb stab 2} is at most $e_m$.  In other words, the parameters $\bs{\beta}$ satisfy the relaxed constraints
\begin{equation}
\sum_{j=1}^{J^*} E_m( Q_{\bs{\beta}^j,\bs{c}^j} ) \leq e_m \quad\tx{for }m=1,\dots,n .
\label{eq:orb stab 3}
\end{equation}
For $m=n+1$, we know that $E_m(q_k)$ converges to the minimum value $C(e_1,\dots,e_n)$, and so RHS\eqref{eq:orb stab 2} is at most $C(e_1,\dots,e_n)$.  This yields
\begin{equation}
\sum_{j=1}^{J^*} E_{n+1}( Q_{\bs{\beta}^j,\bs{c}^j} ) \leq C(e_1,\dots,e_n) .
\label{eq:orb stab 5}
\end{equation}
Strict inequality here should not be possible since $C$ is the minimum value of $E_{n+1}$  over the set~\eqref{eq:con 2}.  Indeed, by~\eqref{eq:conc comp 2} and~\eqref{eq:orb stab 4} we have
\begin{equation*}
\sum_{j=1}^{J^*} E_{n+1}(  Q_{\bs{\beta}^j,\bs{c}^j} )
= \liminf_{k\to\infty} E_{n+1}\bigg( \sum_{j=1}^{J^*} Q_{\bs{\beta}^j,\bs{c}^j}(x-x^j_k) \bigg) 
\geq C(e_1,\dots,e_n) .
\end{equation*}
Therefore, we conclude that equality holds in~\eqref{eq:orb stab 5}.

Altogether, we see that the finite collection $\bs{\beta}$ of parameters is a minimizer for the relaxed variational problem~\eqref{eq:con 2}.  As the minimum value $C(e_1,\dots,e_n)$ is strictly decreasing in each constraint by \cref{thm:min val is dec}, we must have equality in~\eqref{eq:orb stab 3}.  There cannot be $n+1$ distinct values in $\bs{\beta}$, since otherwise we could use \cref{thm:beta mom wiggle} to replace the first $n+1$ distinct values in a way that preserves $E_1,\dots,E_n$ and decreases $E_{n+1}$ in order to obtain a strictly better minimizing sequence.  Now that we know there are at most $n$ distinct values of $\beta_m$, \cref{thm:unique params 2} implies that $\bs{\beta}$ is equal to the unique set of parameters $\beta_1>\cdots>\beta_N>0$ with $N\leq n$ that satisfies the constraints $(e_1,\dots,e_n) \in \M^n_n$.
\end{proof}

It remains to show that the whole sequence $q_k$ converges strongly to the manifold of minimizing multisolitons.  To this end, we will need:
\begin{lem}
\label{thm:min seq lem 3}
The remainders $r^{J^*}_k\to 0$ in $H^n(\R)$ as $k\to\infty$.
\end{lem}
\begin{proof}
By \cref{thm:min seq lem 5}, the profiles $\{ Q_{\bs{\beta}^j,\bs{c}^j} \}_{1\leq j< J^*}$ satisfy the constraints:
\begin{equation*}
\sum_{j=1}^{J^*} E_m( Q_{\bs{\beta}^j,\bs{c}^j} ) = e_m \quad\tx{for }m=1,\dots,n .
\end{equation*}
Combining this with \cref{thm:min seq lem 1,thm:min seq lem 2}, we deduce
\begin{equation*}
0
= \lim_{k\to\infty} E_m\big( r^{J^*}_k \big)
= \lim_{k\to\infty} \tfrac{1}{2} \bnorm{ (r^{J^*}_k)^{(m-1)} }_{L^2}^2
\end{equation*}
for $m=1,\dots,n+1$.
\end{proof}

The last ingredient that we will need is the following ``molecular decomposition'' of multisolitons, which says that our superposition $\sum Q_{\bs{\beta}^j,\bs{c}^j}(x-x^j_k)$ of well-separated multisolitons is close to the manifold of multisolitons:
\begin{prop}
\label{thm:molec decomp}
Fix integers $n\geq 0$ and $J\geq 1$.  Suppose $\bs{\beta}^j$ and $\bs{c}^j$ are multisoliton parameters for each $1 \leq j \leq J$, and that all of the components $\beta^j_m$ of each $\bs{\beta}^j$ are distinct for all $j$ and $m$.  Then for any collection of $J$-many sequences $\{ x^j_k \}_{j=1}^J \subset \R$ satisfying the well-separation condition~\eqref{eq:conc comp 4}, there exists a sequence $\bs{c}_k$ so that
\begin{equation*}
Q_{\bs{\beta},\bs{c}_k}(x) - \sum_{j=1}^{J} Q_{\bs{\beta}^j,\bs{c}^j}(x-x^j_k) \to 0
\quad\tx{in } H^n(\R)\tx{ as } k\to\infty ,
\end{equation*}
where $\bs{\beta}$ is the concatenation $\coprod_{j=1}^{J} \bs{\beta}^j$.
\end{prop}
\begin{proof}
The $n=0$ case is proved in \cite{Killip2020}*{Prop.~3.1}.  Given $n\geq 1$, we pick $\bs{c}_k$ from the $n=0$ case so that the desired convergence occurs in $L^2(\R)$.  Note that
\begin{equation*}
\bigg\lVert Q_{\bs{\beta},\bs{c}_k}(x) - \sum_{j=1}^{J} Q_{\bs{\beta}^j,\bs{c}^j}(x-x^j_k) \bigg\rVert_{H^{n+1}} \lesssim 1
\end{equation*}
uniformly in $k$, by the estimates~\eqref{eq:ests for en cts} that prove that $E_1,\dots,E_{n+2}$ are continuous, the trace formulas~\eqref{eq:en trace}, and the well-separation condition~\eqref{eq:conc comp 4}.  Using the inequality
\begin{equation*}
\norm{f}_{H^n} \leq \norm{f}_{L^2}^{\frac{1}{n+1}} \norm{f}_{H^{n+1}}^{\frac{n}{n+1}}
\end{equation*}
(which follows from H\"older's inequality in Fourier variables), we conclude that the sequence converges in $H^n(\R)$.
\end{proof}

We are now prepared to finish the proof of \cref{thm:min seq}:
\begin{proof}[Proof of \cref{thm:min seq}]
It remains to show that the sequence $\{q_k\}_{k\geq 1}$ converges to the manifold $\{ Q_{\bs{\beta},\bs{c}} : \bs{c} \in \R^N \}$.  So far, we have the decomposition
\begin{equation*}
q_k(x) = \sum_{j=1}^{J^*} Q_{\bs{\beta}^j,\bs{c}^j}(x-x^j_k) + r^{J^*}_k(x)
\end{equation*}
with $J^*$ finite.  Let $Q_{\bs{\beta},\bs{c}_k}$ be the sequence of multisolitons guaranteed by the $H^n(\R)$ molecular decomposition (\cref{thm:molec decomp}).  We estimate
\begin{equation*}
\snorm{ q_k - Q_{\bs{\beta},\bs{c}_k} }_{H^n}
\leq \bigg\lVert Q_{\bs{\beta},\bs{c}_k}(x) - \sum_{j=1}^{J^*} Q_{\bs{\beta}^j,\bs{c}^j}(x-x^j_k) \bigg\rVert_{H^n} + \bnorm{ r^{J^*}_k }_{H^n} .
\end{equation*}
The first term on the RHS converges to zero as $k\to\infty$ by \cref{thm:molec decomp}.  The second term on the RHS converges to zero by \cref{thm:min seq lem 3}.  Together, we conclude that
\begin{equation*}
\inf_{\bs{c}\in\R^n} \snorm{ q_k - Q_{\bs{\beta},\bs{c}} }_{H^n} \to 0 \quad\tx{as}\quad k\to\infty
\end{equation*}
as desired.
\end{proof}

As a corollary, we obtain orbital stability:
\begin{proof}[Proof of \cref{thm:orb stab}]
Suppose towards a contradiction that orbital stability fails.  Then there exists a constant $\eps_0>0$, a sequence of initial data $\{ q_{k}(0) \}_{k\geq 1} \subset H^n(\R)$, and a sequence of times $\{t_k\}_{k\geq 1}\subset\R$ such that
\begin{equation*}
\inf_{\bs{c}\in\R^n} \snorm{ q_{k}(0) - Q_{\bs{\beta},\bs{c}} }_{H^n} \to 0 \quad\tx{as } k\to\infty ,
\end{equation*}
but the corresponding solutions $q_k(t)$ to KdV obey
\begin{equation}
\inf_{\bs{c}\in\R^n} \snorm{ q_k(t_k) - Q_{\bs{\beta},\bs{c}} }_{H^n} \geq \eps_0 \quad\tx{for all } k.
\label{eq:orb stab 1}
\end{equation}

As $E_1,\dots,E_{n+1}$ are continuous on $H^n(\R)$ and are conserved by the KdV flow, we have
\begin{equation*}
\lim_{k\to\infty} E_m(q_k(t_k))
= \lim_{k\to\infty} E_m(q_{k}(0))
= E_m(Q_{\bs{\beta},\bs{c}})
\end{equation*}
for each $m=1,\dots,n+1$.  There are $n$-many $\beta$ parameters, and so these are exactly the conditions~\eqref{eq:min seq 1} that the sequence $\{q_k(t_k)\}_{k\geq 1}$ is a minimizing sequence for $E_{n+1}$ with constraints $E_1(Q_{\bs{\beta},\bs{c}}),\dots, E_n(Q_{\bs{\beta},\bs{c}})$ that are in $\M^n_n$.  By \cref{thm:min seq}, there exists a subsequence of $\{q_k(t_k)\}_{k\geq 1}$ which converges to the manifold $\{ Q_{\bs{\beta},\bs{c}} : \bs{c}\in\R^N \}$ in $H^n(\R)$, which contradicts our assumption~\eqref{eq:orb stab 1}.
\end{proof}


\section{Proof of \cref{thm:N gas}}
\label{sec:N gas}

In this section, we will adapt the methods of \cref{sec:eng,sec:min,sec:orb stab} in order to prove \cref{thm:N gas}.  Fix $N\geq n+1$, and consider constraints $(e_1,\dots,e_n) \in \M^n_{N}$ that are attainable by a multisoliton of degree at most $N$.  We aim to show that minimizing sequences resemble a superposition of multisolitons with at most $n$ distinct amplitudes.

As the constraints are attainable by finitely many parameters, compactness still guarantees that there exists a minimizing set of $N$-soliton parameters $\beta_1\geq \cdots \geq \beta_N \geq 0$, provided that we allow for repeated values:
\begin{lem}
\label{thm:N gas lem 1}
Given constraints $(e_1,\dots,e_n) \in \M^n_{N}$, there exist $\beta_1\geq \cdots\geq\beta_N \geq 0$ which minimize
\begin{equation*}
E_{n+1}(Q_{\bs{\beta},\bs{c}}) = (-1)^n \tfrac{2^{2n+3}}{2n+3} \sum_{m=1}^N \beta_m^{2n+3}
\end{equation*}
over the set of multisolitons in the constraint set $\con$.
\end{lem}
\begin{proof}
Consider the set $\Gamma$ of parameters
\begin{equation}
\bigg\{ (x_1,\dots,x_N) \in\R^N : x_1,\dots,x_N\geq 0,\ \sum_{m=1}^N x_m^{2j+1} = \alpha_j \tx{ for }j=1,\dots,n \bigg\} 
\label{eq:N gas 2}
\end{equation}
that satisfy the constraints, where
\begin{equation*}
\alpha_j = (-1)^{j+1}\tfrac{2j+1}{2^{2j+1}} e_j
\end{equation*}
are the prescribed odd moments.  Note that the set $\Gamma$ is compact, and it is nonempty since $(e_1,\dots,e_n) \in \M^n_N$.  Therefore there exists a minimizer $(\beta_1,\dots,\beta_N)$ of the next odd moment
\begin{equation*}
(-1)^n \tfrac{2^{2n+3}}{2n+3} \sum_{m=1}^N x_m^{2n+3}
\end{equation*}
in $\Gamma$.  As the odd moments are symmetric in $\beta_1,\dots,\beta_N$, we may reorder them so that $\beta_1\geq \cdots \geq \beta_N \geq 0$.
\end{proof}

Unlike in the proof of \cref{thm:down closed}, the set $\Gamma$ no longer reaches the boundary $\{ (x_1,\dots,x_N) : x_{n+1},\dots,x_N = 0\}$ when $(e_1,\dots,e_n)\notin \M^n_n$, and so we cannot necessarily reduce the number of parameters.  Instead, we can employ the implicit function theorem argument from \cref{thm:beta mom wiggle} to reduce the number of distinct components in the minimizer $(\beta_1,\dots,\beta_N)$:
\begin{lem}
\label{thm:N gas lem 2}
If $\beta_1\geq \cdots\beta_N \geq 0$ is a minimizer (in the sense of \cref{thm:N gas lem 1}), then there are at most $n$ distinct values of $\beta_m$.
\end{lem}
\begin{proof}
It suffices to show that if there are at least $n+1$ distinct values in $\beta_1,\dots,\beta_N$, then there exist new values $\wt{\beta}_1,\dots,\wt{\beta}_N$ which preserve $E_1,\dots,E_n$ but decrease $E_{n+1}$.  To prove this, we repeat the proof of \cref{thm:beta mom wiggle}.  Rather than recapitulating the whole proof, let us focus on the few minor alterations that need to be made.

For example, consider the case where we have $N\geq n+2$, $\beta_j = \beta_{j+1}$ for some $j$, and all other $\beta_m$ are distinct.  Replace the function~\eqref{eq:beta mom wiggle 6} by
\begin{equation*}
f(x_1,\dots,x_{n+1}) = \begin{pmatrix}
x_1^3 + \dots + x_{j-1}^3 + 2x_j^3 + x_{j+1}^3 + \dots + x_{n+1}^3 \\
x_1^5 + \dots + x_{j-1}^5 + 2x_j^5 + x_{j+1}^5 + \dots + x_{n+1}^5 \\
\vdots \\
x_1^{2n+1} + \dots + x_{j-1}^{2n+1} + 2x_j^{2n+1} + x_{j+1}^{2n+1} + \dots + x_{n+1}^{2n+1}
\end{pmatrix} .
\end{equation*}
This simply multiplies the $j$th column of the derivative matrix~\eqref{eq:beta mom wiggle 7} by $2$.  Consequently, the left $n\times n$ submatrix still has nonzero determinant and thus we may apply the implicit function theorem.  We can then proceed with the remainder of the proof of \cref{thm:beta mom wiggle}.

In the general case, each column of the derivative matrix~\eqref{eq:beta mom wiggle 7} is simply multiplied by a constant.  Therefore the left $n\times n$ submatrix is still invertible, and the proof of \cref{thm:beta mom wiggle} proceeds as before.
\end{proof}

Now that we know that every minimizer must possess at most $n$ distinct $\beta$ values, \cref{thm:unique params 2} immediately implies that the minimizer is unique.

We are now prepared to define our candidate value for the infimum of $E_{n+1}$ subject to the constraints $e_1,\dots,e_n$.    We extend the definition of $C$ to $(e_1,\dots,e_n)\in \M^n_N$ via
\begin{equation*}
C(e_1,\dots,e_n)
= (-1)^n \tfrac{2^{2n+3}}{2n+3} \sum_{j=1}^N \beta_j^{2n+3} .
\end{equation*}
This quantity still satisfies the properties from \cref{thm:min val is dec}:
\begin{lem}
\label{thm:min val is dec 2}
The function $C:\M^n_N\to\R$ is continuous and is decreasing in each variable.  Moreover, $C$ is defined piecewise on finitely many connected subsets of $\M^n_N$, and on the interior of each such subset $C(e_1,\dots,e_n)$ is continuously differentiable and satisfies $\frac{\partial C}{\partial e_j} < 0$ for $j=1,\dots,n$.
\end{lem}
\begin{proof}
Given a minimizer $\beta_1\geq \cdots\geq \beta_N \geq 0$, \cref{thm:N gas lem 2} implies that there exist multiplicities $m_1,\dots,m_{\ol{N}}$ and distinct values $\ol{\beta}_1>\cdots>\ol{\beta}_{\ol{N}}\geq 0$ so that $\ol{N}\leq n$, $\sum m_j = N$, and the string $\beta_1,\dots,\beta_N$ consists of $m_1$ copies of $\ol{\beta}_1$, $m_2$ copies of $\ol{\beta}_2$, and so on.  This allows us to write
\begin{equation*}
C(e_1,\dots,e_n) = (-1)^n \tfrac{2^{2n+3}}{2n+3} \big( m_1\ol{\beta}_1^{2n+3} + \dots + m_{\ol{N}} \ol{\beta}_{\ol{N}}^{2n+3} \big) .
\end{equation*}
We will see that for $\ol{N} = n$ and each fixed choice of multiplicities $m_1,\dots,m_{n}$, we have $\frac{\partial C}{\partial e_j} < 0$ for $j=1,\dots,n$ as long as $\ol{\beta}_1>\cdots>\ol{\beta}_{\ol{N}}> 0$.  In this way $C(e_1,\dots,e_n)$ is a piecewise-defined function, and there are finitely many pieces because the number of possible multiplicities $m_1,\dots,m_{\ol{N}}$ is finite.  

Fix multiplicities $m_1,\dots,m_{\ol{N}}$, and repeat the computation from \cref{thm:min val is dec}.  In fact, in the case $\ol{N} = n$ the same computation applies!  Indeed, in \cref{thm:min val is dec} we computed $\frac{\partial C}{\partial e_j}$ from the equality~\eqref{eq:min val is dec 4}.  We have now multiplied the columns of the matrix on the LHS and each entry on the RHS by the multiplicities $m_1,\dots,m_n$, but this does not alter the system of equations.  In the case $\ol{N}<n$, the system of equations~\eqref{eq:min val is dec 1} is overdetermined.  However, if we only consider the first $\ol{N}$ constraints, then the computation proceeds with $\ol{N}$ in place of $n$, and we conclude that $C$ as a function of $e_1,\dots,e_{\ol{N}}$ (where $m_1,\dots,m_{\ol{N}}$ are fixed) is $C^1$ and satisfies $\frac{\partial C}{\partial e_j} < 0$ for $j=1,\dots,\ol{N}$.

We are also able to compute $\frac{\partial \ol{\beta}_k}{\partial e_j}$ as long as $\ol{\beta}_1>\cdots>\ol{\beta}_n> 0$, which implies that if we wiggle $e_1,\dots,e_n$ then we can also wiggle $\ol{\beta}_1,\dots,\ol{\beta}_n$ in a way that still satisfies the constraints.  By uniqueness (\cref{thm:unique params 2}), the perturbed values of $\ol{\beta}_1,\dots,\ol{\beta}_n$ still minimize $E_{n+1}$.  This defines an injective map $\Phi$ from the simplex
\begin{equation}
\{ (\ol{\beta}_1,\dots,\ol{\beta}_n) \in \R^{n} : \ol{\beta}_1> \dots > \ol{\beta}_{n} > 0 \}
\label{eq:simplex 2}
\end{equation}
into $\M^n_n$, and it is smooth up to its boundary.  The image of $\Phi$ is exactly the interior of one of the components on which $C(e_1,\dots,e_n)$ is defined by a single formula.  The boundary of this component corresponds to some subset of the boundary of the simplex~\eqref{eq:simplex 2}, which means that two values of $\ol{\beta}_j$ are colliding or that $\ol{\beta}_n$ is vanishing.  Repeating the proof of \cref{thm:homeo}, we conclude that $\Phi$ is a homeomorphism onto this component (including any boundary points it may contain).  It then follows that $C$ is continuous by the same argument as in \cref{thm:min val is dec}.
\end{proof}

Next, we show that the set of constraints $\M^n_N$ is still downward closed:
\begin{lem}
\label{thm:down closed 2}
If the constraints $\wt{e}_1,\dots,\wt{e}_n$ are in $\M^n_{\wt{N}}$ for some $\wt{N}$ and
\begin{equation*}
\wt{e}_1\leq e_1 , \quad\dots,\quad \wt{e}_n\leq e_n
\end{equation*}
for some $(e_1,\dots,e_n) \in \M^n_N$, then $(\wt{e}_1,\dots,\wt{e}_n) \in \M^n_N$.
\end{lem}
\begin{proof}
Let $\wt{\beta}_1\dots,\wt{\beta}_{\wt{N}} > 0$ denote the $\beta$ parameters of the multisoliton which witnesses the constraints $\wt{e}_1,\dots,\wt{e}_n$, and assume that we are in the nontrivial case $\wt{N} \geq N+1$.  Repeating the proof of \cref{thm:down closed}, we conclude that the set $\Gamma$ of parameters in $\R^{\wt{N}}$ that satisfy the constraints must intersect the boundary $\{(x_1,\dots,x_{\wt{N}}) : x_{N+1},\dots,x_{\wt{N}} = 0\}$.  Any point in the intersection provides the desired $N$-soliton parameters.
\end{proof}

We are now prepared to prove that $C(e_1,\dots,e_n)$ is the infimum of $E_{n+1}$:
\begin{prop}
\label{thm:minimizers 2}
Given $(e_1,\dots,e_n) \in \M^n_{N}$ for some $N\geq n+1$, we have
\begin{equation}
\inf \{ E_{n+1}(u) : u\in\con \} = C(e_1,\dots,e_n) .
\label{eq:N gas 1}
\end{equation}
Moreover, if $(e_1,\dots,e_n) \notin\M^n_n$, then this infimum is not attained by any $u\in\con$.
\end{prop}
\begin{proof}
First, we claim that
\begin{equation*}
E_{n+1}(u) \geq C(e_1,\dots,e_n) \quad\tx{for all }u\in\con. 
\end{equation*}
We repeat the proof of \cref{thm:minimizers}.  This proof only required \cref{thm:beta mom wiggle,thm:min val is dec,thm:down closed} as input, and we have established their analogues \cref{thm:N gas lem 2,thm:min val is dec 2,thm:down closed 2} in this new setting.

To prove~\eqref{eq:N gas 1}, it remains to show that $E_{n+1}(u)$ can be arbitrarily close to $C(e_1,\dots,e_n)$ for some choice of $u\in\con$.  Recall that $C(e_1,\dots,e_n)$ is defined in terms of a minimizer $\beta_1\geq\cdots\geq\beta_N \geq 0$ in the sense of \cref{thm:N gas lem 1}.  The claim follows by taking $u$ to be an $N$-soliton with parameters $\wt{\beta}_1>\cdots>\wt{\beta}_N>0$ that converge to the minimizer $\beta_1\geq\cdots\geq\beta_N \geq 0$ within the set~\eqref{eq:N gas 2}.

Lastly, suppose towards a contradiction that $E_{n+1}(q) = C(e_1,\dots,e_n)$ for some $q \in \con$ and $(e_1,\dots,e_n) \notin\M^n_n$.  First, we show that $q$ is Schwartz by repeating the proof of \cref{thm:schwartz}.  In the case where the number $\ol{N}$ of multiplicities is equal to $n$, we have $\frac{\partial C}{\partial e_j} < 0$ for $j=1,\dots,n$ and the proof of \cref{thm:schwartz} carries out unaltered.  In the case where $\ol{N} < n$, we recall from the proof of \cref{thm:min val is dec 2} that we may regard $C$ as a function of $e_1,\dots,e_{\ol{N}}$ and we have $\frac{\partial C}{\partial e_j} < 0$ for $j=1,\dots,\ol{N}$.  In either case, the minimizer $q$ satisfies an Euler--Lagrange equation of the form~\eqref{eq:schwartz 1} with $\lambda_1 < 0$ and all other $\lambda_j \leq 0$, and this is sufficient to conclude that $q$ is Schwartz.  Then, by directly applying \cref{thm:a mom wiggle}, \ref{thm:inner outer fac}, and \ref{thm:beta mom wiggle} (without alteration!), we see that $q$ is a multisoliton of degree at most $n$, which contradicts that $(e_1,\dots,e_n) \notin \M_{n}^n$.
\end{proof}

When combined with concentration compactness, we can prove that minimizing sequences resemble a superposition of multisolitons with at most $n$ distinct values of $\beta_m$:
\begin{proof}[Proof of \cref{thm:N gas}]
It only remains to prove the minimizing sequence statement.  Fix $(e_1,\dots,e_n) \in \M^n_{N}\sm\M^n_{N-1}$ for some $N\geq n+1$, and suppose that $\{q_k\}_{k\geq 1} \subset H^n(\R)$ satisfies
\begin{equation*}
E_1(q_k) \to e_1 , \quad\dots,\quad E_n(q_k) \to e_n , \quad E_{n+1}(q_k) \to C(e_1,\dots,e_n)
\end{equation*}
as $k\to\infty$.  

First, we apply our concentration compactness principle.  After passing to a subsequence, \cref{thm:conc comp} provides us with a number $J^*\in\{0,1,\dots,\infty\}$, $J^*$-many profiles $\{\phi^j\}_{j=1}^{J^*} \subset H^n(\R)$, and $J^*$-many sequences $\{ x^j_k \}_{j=1}^{J^*} \subset \R$ so that along a subsequence we have the decomposition~\eqref{eq:conc comp decomp} which satisfies the properties~\eqref{eq:conc comp 1}-\eqref{eq:conc comp 4}.

The proof of \cref{thm:min seq} up through \cref{thm:min seq lem 4} still applies (without alteration), and so we conclude that each profile $\phi^j$ is a multisoliton $Q_{\bs{\beta}^j,\bs{c}^j}$.  Repeating the proof of \cref{thm:min seq lem 5}, we see that the concatenation $\bs{\beta} = \coprod_{j=1}^{J^*} \bs{\beta}^j$ minimizes the inequality for $E_{n+1}$ in \cref{thm:minimizers 2}.  Therefore $\bs{\beta}$ is a minimizer in the sense of \cref{thm:N gas lem 1}, and so by \cref{thm:N gas lem 2,thm:down closed 2} we see that $J^*$ is finite, the total degree $\sum \#\bs{\beta}^j$ is equal to $N$, and the components of $\bs{\beta}$ attain at most $n$ distinct values.

We now have
\begin{equation*}
\bigg\lVert q_k - \sum_{j=1}^{J^*} Q_{\bs{\beta}^j,\bs{c}^j} \bigg\rVert_{H^n}
= \bnorm{ r^{J^*}_k }_{H^n} .
\end{equation*}
Repeating the proof of \cref{thm:min seq lem 3}, we see that the RHS converges to zero as $k\to\infty$.  This yields
\begin{equation*}
\inf_{\bs{c}^1,\dots,\bs{c}^{J^*}} \bigg\lVert q_k - \sum_{j=1}^{J^*} Q_{\bs{\beta}^j,\bs{c}^j} \bigg\rVert_{H^n} \to 0 \quad\tx{as }k\to\infty
\end{equation*}
as desired.
\end{proof}


\section{Proof of \cref{thm:point mass}}
\label{sec:point mass}

The goal of this section is to prove \cref{thm:point mass}.  Suppose that $(e_1,e_2) \in \F^2$ and $(e_1,e_2) \notin \M^2_N$ for all $N$.  We aim to show that Schwartz minimizing sequences for these constraints have vanishing $\beta$ parameters and $\log|a|$ converging to the even extension of a Dirac delta distribution.

By the explicit description~\eqref{eq:M2 F2} and~\eqref{eq:MN} of $\F^2$ and $\bigcup_{N\geq 0} \M^2_N$, we note that our conditions on $(e_1,e_2)$ are equivalent to $e_1>0$ and $e_2\geq 0$.

First, we find a lower bound for the $\log|a|$ contribution to $E_3$:
\begin{lem}
\label{thm:point mass lem 1}
We have
\begin{equation}
E_3(u) \geq \tfrac{64}{\pi} \tfrac{\gamma_1^2}{\gamma_0} + C\big(e_1 - \tfrac{4}{\pi} \gamma_0 , e_2 - \tfrac{16}{\pi} \gamma_1 \big) \quad\tx{for all }u\in \con\cap\schwar(\R),
\label{eq:point mass 1}
\end{equation}
where
\begin{equation}
\gamma_0 = \int_{-\infty}^\infty k^2 \log|a(k;u)|\dk 
\quad\tx{and}\quad
\gamma_1 = \int_{-\infty}^\infty k^4 \log|a(k;u)|\dk . 
\label{eq:point mass 7}
\end{equation}
\end{lem}
\begin{proof}
Fix $u \in \schwar(\R)$.  Substituting $x = k^2$ into the trace formulas~\eqref{eq:en trace} and recalling the reality condition~\eqref{eq:trans coeff 3}, we obtain
\begin{align*}
e_1 &= \tfrac{4}{\pi} \int_0^\infty x^{\frac{1}{2}} \log|a(x^{\frac{1}{2}};u)|\dx + \tfrac{8}{3} \sum_{m=1}^N \beta_m^3 , \\ 
e_2 &= \tfrac{16}{\pi} \int_0^\infty x^{\frac{3}{2}} \log|a(x^{\frac{1}{2}};u)|\dx - \tfrac{32}{5} \sum_{m=1}^N \beta_m^5 , \\
E_3(u) &= \tfrac{64}{\pi} \int_0^\infty x^{\frac{5}{2}} \log|a(x^{\frac{1}{2}};u)|\dx + \tfrac{128}{7} \sum_{m=1}^N \beta_m^7 .
\end{align*}
The first constraint says that the positive measure
\begin{equation*}
\dd\mu := x^{\frac{1}{2}} \log|a(x^{\frac{1}{2}};u)|\dx
\end{equation*}
on $[0,\infty)$ has total mass
\begin{equation}
\gamma_0 := \int_0^\infty 1\dmu(x) = \tfrac{\pi}{4} \bigg( e_1 - \tfrac{8}{3} \sum_{m=1}^N \beta_m^3 \bigg) 
\in \big[ 0,\tfrac{\pi}{4} e_1 \big] .
\label{eq:point mass 4}
\end{equation}

The first constraint also restricts how large the first moment of $\dd\mu$ can be.  As $p\mapsto \snorm{ \bs{\beta} }_{\ell^p}$ is decreasing, we have
\begin{equation*}
\bigg( \sum_{m=1}^N \beta_m^5 \bigg)^{\!\frac{1}{5}}\! \leq \bigg( \sum_{m=1}^N \beta_m^3 \bigg)^{\!\frac{1}{3}}\! \leq \big(\tfrac{3}{8} e_1\big)^{\frac{1}{3}} .
\end{equation*}
This requires that the first moment obeys
\begin{equation}
\gamma_1 := \int_0^\infty x \dmu(x) = \tfrac{\pi}{16} \bigg( e_2 + \tfrac{32}{5} \sum_{m=1}^N \beta_m^5 \bigg)
\in \big[ \tfrac{\pi}{16} e_2,\ \tfrac{\pi}{16} \big( e_2 + \tfrac{32}{5} (\tfrac{3}{8}e_1)^{\frac{5}{3}} \big) \big] .
\label{eq:point mass 5}
\end{equation}

In order to bound $E_3(u)$ below, we seek a lower bound for the second moment 
\begin{equation*}
\gamma_2 := \int_0^\infty x^2 \dmu(x) .
\end{equation*}
By Cauchy--Schwarz we have
\begin{equation*}
\gamma_1 
= \int_0^\infty x \dmu(x) 
\leq \bigg( \int_0^\infty 1 \dmu(x) \bigg)^{\!\frac{1}{2}} \bigg( \int_0^\infty x^2 \dmu(x) \bigg)^{\!\frac{1}{2}}
= \gamma_0^{\frac{1}{2}} \gamma_2^{\frac{1}{2}} ,
\end{equation*}
and so
\begin{equation}
\gamma_2 \geq \tfrac{\gamma_1^2}{\gamma_0} .
\label{eq:point mass 3}
\end{equation}	
For future reference (cf.~\eqref{eq:point mass 6}), we note that equality occurs above if and only if the functions $x^{\frac{1}{2}} \log|a(x^{\frac{1}{2}};u)|$ and $x^{\frac{5}{2}} \log|a(x^{\frac{1}{2}};u)|$ are proportional.  As $k\mapsto \log|a(k;u)|$ is a continuous nonnegative function on $\R$ for $u$ Schwartz, this can only happen when $\log|a(k;u)| \equiv 0$ for $k\in\R$.

The estimate~\eqref{eq:point mass 3} provides a lower bound for the $\log|a|$ moment of $E_3(u)$.  The $\beta$ moment is then bounded below by the infimum $C$ of $E_3(u)$ subject to the smaller constraints where the $\log|a|$ moments are removed:
\begin{equation*}
\tfrac{128}{7} \sum_{m=1}^N \beta_m^7 
= E_3(Q_{\bs{\beta},\bs{c}}) 
\geq C\big(e_1 - \tfrac{4}{\pi} \gamma_0 , e_2 - \tfrac{16}{\pi} \gamma_1 \big) .
\end{equation*}
Together, this yields the inequality~\eqref{eq:point mass 1}.
\end{proof}

We will now minimize the lower bound in~\eqref{eq:point mass 1} over all possible $\gamma_0$ and $\gamma_1$.  At first glance, the first term $\gamma_1^2 / \gamma_0$ is smallest when $\gamma_0$ is large and $\gamma_1$ is small, while the second term $C(e_1 - \tfrac{4}{\pi} \gamma_0 , e_2 - \tfrac{16}{\pi} \gamma_1 )$ is smallest when both $\gamma_0$ and $\gamma_1$ are small.  We will see below that the first term is dominant, which yields the following inequality:
\begin{lem}
\label{thm:point mass lem 2}
Given constraints $e_1>0$ and $e_2\geq 0$, we have
\begin{equation}
\inf \{ E_3(u) : u\in \con\cap\schwar(\R) \} = \tfrac{e_2^2}{e_1} .
\label{eq:point mass 8}
\end{equation}
Moreover, this infimum is not attained by any $u\in \con\cap\schwar(\R)$.
\end{lem}
\begin{proof}
The domain of $(\gamma_0,\gamma_1)$ in $\R^2$ is contained in the rectangle given by the product of the intervals in~\eqref{eq:point mass 4} and~\eqref{eq:point mass 5}.  Let $(\gamma_0,\gamma_1)$ be the minimizer of
\begin{equation*}
\tfrac{64}{\pi} \tfrac{\gamma_1^2}{\gamma_0} + C(e_1 - \tfrac{4}{\pi} \gamma_0 , e_2 - \tfrac{16}{\pi} \gamma_1 )
\end{equation*}
over this compact rectangle.  Differentiating with respect to $\gamma_1$, we have
\begin{equation*}
\tfrac{\partial}{\partial\gamma_1} \big\{ \tfrac{64}{\pi} \tfrac{\gamma_1^2}{\gamma_0} + C( e_1 - \tfrac{4}{\pi} \gamma_0 , e_2 - \tfrac{16}{\pi} \gamma_1 ) \big\}
= \tfrac{128}{\pi} \tfrac{\gamma_1}{\gamma_0} -\tfrac{16}{\pi} \tfrac{\partial C}{\partial e_2}( e_1 - \tfrac{4}{\pi} \gamma_0 , e_2 - \tfrac{16}{\pi} \gamma_1 ) .
\end{equation*}
The derivative $\frac{\partial C}{\partial e_2}$ is nonpositive by \cref{thm:min val is dec}, and so this quantity is positive for all $\gamma_1$ in the open interval $( \tfrac{\pi}{16} e_2,\ \tfrac{\pi}{16} ( e_2 + \tfrac{32}{5} (\tfrac{3}{8}e_1)^{\frac{5}{3}} ) )$.  Therefore the minimizer must have $\gamma_1 = \tfrac{\pi}{16} e_2$.  Similarly, we have
\begin{equation*}
\tfrac{\partial}{\partial\gamma_0} \big\{ \tfrac{64}{\pi} \tfrac{\gamma_1^2}{\gamma_0} + C( e_1 - \tfrac{4}{\pi} \gamma_0 , e_2 - \tfrac{16}{\pi} \gamma_1 ) \big\}
= -\tfrac{64}{\pi} \tfrac{\gamma_1^2}{\gamma_0^2} -\tfrac{4}{\pi} \tfrac{\partial C}{\partial e_1}( e_1 - \tfrac{4}{\pi} \gamma_0 , e_2 - \tfrac{16}{\pi} \gamma_1 ) .
\end{equation*}
The derivative $\frac{\partial C}{\partial e_1}$ is $O(\beta_1^2\beta_2^2)$ by the computation~\eqref{eq:min val is dec 5}, and hence vanishes as $\beta_1,\beta_2\to 0$.  Therefore, taking $\gamma_1 \to \tfrac{\pi}{16} e_2$ we obtain
\begin{equation*}
\tfrac{\partial}{\partial\gamma_0} \big\{ \tfrac{64}{\pi} \tfrac{\gamma_1^2}{\gamma_0} + C( e_1 - \tfrac{4}{\pi} \gamma_0 , e_2 - \tfrac{16}{\pi} \gamma_1 ) \big\}
\to - \tfrac{\pi}{4}\tfrac{e_2^2}{\gamma_0^2}
\end{equation*}
for all $\gamma_0$ in the open interval $( 0,\tfrac{\pi}{4} e_1 )$.  Therefore the minimizer has $\gamma_0 = \tfrac{\pi}{4} e_1$.

Altogether, we conclude that the minimum occurs at $\gamma_0 = \frac{\pi}{4} e_1$, $\gamma_1 = \frac{\pi}{16}e_2$ with value 
\begin{equation*}
\big\{ \tfrac{64}{\pi} \tfrac{\gamma_1^2}{\gamma_0} + C(e_1 - \tfrac{4}{\pi} \gamma_0 , e_2 - \tfrac{16}{\pi} \gamma_1 ) \big\} \Big|_{\gamma_0 = \frac{\pi}{4} e_1,\ \gamma_1 = \frac{\pi}{16}e_2} = \tfrac{e_2^2}{e_1} .
\end{equation*}	
To prove~\eqref{eq:point mass 8}, it remains to show that we can make $E_3(u)$ arbitrarily close to this value.  Fix $(\wt{\gamma}_0,\wt{\gamma}_1)$ in the interior of the rectangle given by the product of the intervals in~\eqref{eq:point mass 4} and~\eqref{eq:point mass 5} that is arbitrarily close to the minimizer $(\gamma_0,\gamma_1)$.  Pick a smooth and even function $k\mapsto\log|a(k;\wt{u})|$ with compact support in $\R\sm\{0\}$ which attains the moments $(\wt{\gamma}_0,\wt{\gamma}_1)$ (in the sense of~\eqref{eq:point mass 7}).  Arguing as in \cref{thm:a mom wiggle}, we can then use \cref{thm:inv scat thy} to construct a function $\wt{u} \in \schwar(\R)$ (with no bound states) so that $k\mapsto\log|a(k;\wt{u})|$ attains the prescribed moments $(\wt{\gamma}_0,\wt{\gamma}_1)$.

Lastly, suppose that
\begin{equation}
E_3(q) = \tfrac{e_2^2}{e_1}
\label{eq:point mass 9}
\end{equation}
for some $q\in\con\cap\schwar(\R)$.  Then we would have $\gamma_0 = \frac{\pi}{4} e_1$, $\gamma_1 = \frac{\pi}{16}e_2$ and hence the $\beta$ moments $\sum \beta_m^3$ and $\sum \beta_m^5$ must vanish.  As $e_1>0$ then this implies $\log |a(k;u)| \not\equiv 0$, and so we must have strict inequality in~\eqref{eq:point mass 3}:
\begin{equation}
E_3(q) = \tfrac{64}{\pi} \gamma_2 > \tfrac{64}{\pi} \tfrac{\gamma_1^2}{\gamma_0} = \tfrac{e_2^2}{e_1} .
\label{eq:point mass 6}
\end{equation}
This contradicts the assumption~\eqref{eq:point mass 9}, and so such a minimizer $q$ cannot exist.
\end{proof}

Now that we have found the infimum of $E_3$, we are prepared to analyze Schwartz minimizing sequences:
\begin{proof}[Proof of \cref{thm:point mass}]
Fix a minimizing sequence $\{q_j\}_{j\geq 1} \subset \schwar(\R)$, so that
\begin{equation}
E_1(q_j) \to e_1 , \quad E_2(q_j) \to e_2 , \quad E_{3}(q_j) \to \tfrac{e_2^2}{e_1} \quad\tx{as }j\to\infty .
\label{eq:point mass 2}
\end{equation}
In the inequality of \cref{thm:point mass lem 2}, we see that we have equality in the limit $j\to\infty$.  Therefore the moments $\sum_{m\geq 1} \beta_{j,m}^7$ vanish as $j\to\infty$; otherwise, we could construct a strictly better minimizing sequence with no $\beta$ parameters, because the constraints can be met solely in terms of the $\log |a|$ moments.  This implies
\begin{equation*}
\beta_{j,m} \leq \bigg( \sum_{\ell\geq 1} \beta_{j,\ell}^7 \bigg)^{\!\frac{1}{7}}\! \to 0\quad\tx{as }j\to\infty
\end{equation*}
for all $m$.

We pass to an arbitrary subsequence of $\{q_j\}_{j\geq 1}$.  We claim that there is a further subsequence with $\log|a|\dk$ converging to the even extension of a unique point mass, from which it will follow that the whole sequence $\log|a(k;q_j)|\dk$ converges to the same limit.  Consider the measures
\begin{equation*}
d\mu_j := x^{\frac{1}{2}} \log|a(x^{\frac{1}{2}};q_j)|\dx
\end{equation*}
on $[0,\infty)$.  Changing variables $x = k^2$, the convergence~\eqref{eq:point mass 2} together with the trace formulas~\eqref{eq:en trace} and the reality condition~\eqref{eq:trans coeff 3} tell us that moments of $d\mu_j$ obey
\begin{align*}
\gamma_0^j &:= \int_0^\infty 1\, d\mu_j(x) \to \tfrac{\pi}{4} e_1 =: \gamma_0 , \\
\gamma_1^j &:= \int_0^\infty x\, d\mu_j(x) \to \tfrac{\pi}{16} e_2 =: \gamma_1 ,\\
\gamma_2^j &:= \int_0^\infty x^2\, d\mu_j(x) \to \tfrac{\gamma_1^2}{\gamma_0} .
\end{align*}

We claim that the renormalized measures $d\mu_j / \gamma_0^j$ on $[0,\infty)$ are tight.  For $R>0$ we estimate
\begin{equation*}
\tfrac{1}{\gamma_0^j} \mu_j((R,\infty))
= \tfrac{1}{\gamma_0^j} \int_R^\infty d\mu_j(x)
\leq \tfrac{1}{R\gamma_0^j}\int_R^\infty x\, d\mu_j(x) .
\end{equation*}
Note that $1/\gamma_0^j$ is bounded uniformly for $j$ large since $\gamma_0^j \to \gamma_0 > 0$.  Also, the integral on the RHS is bounded uniformly in $j$ since $\gamma_1^j \to \gamma_1$.  Together, we conclude that the RHS tends to zero as $R\to\infty$ uniformly in $j$.

Therefore, by Prokhorov's theorem we may pass to a subsequence along which the probability measures $d\mu_j / \gamma_0^j$ converge weakly to some probability measure $d\mu/\gamma_0$.  As $\gamma_0^j \to \gamma_0 > 0$, then the measures $d\mu_j$ converge weakly to $d\mu_j$.

The sequence of second moments $\gamma_2^j$ converges, and hence is bounded.  It then follows that the zeroth and first moments converge to those of $\mu$:
\begin{equation*}
\int_0^\infty d\mu_j(x) = \lim_{j\to\infty} \int_0^\infty d\mu_j(x) = \gamma_0 , \quad
\int_0^\infty x\, d\mu_j(x) = \lim_{j\to\infty} \int_0^\infty x\, d\mu_j(x) = \gamma_1.
\end{equation*}
For the second moments, we use Fatou's lemma (which holds for weakly converging measures) to obtain
\begin{equation*}
\tfrac{\gamma_1^2}{\gamma_0} \leq \int_0^\infty x^2\, d\mu(x) \leq \liminf_{j\to\infty} \int_0^\infty x^2\, d\mu_j(x) = \tfrac{\gamma_1^2}{\gamma_0} .
\end{equation*}
Altogether, we conclude that $\mu$ minimizes the second moment lower bound~\eqref{eq:point mass 3} from the Cauchy--Schwarz inequality.  Therefore the distributions $d\mu$ and $x^2\dmu(x)$ on $[0,\infty)$ are proportional, and hence $\mu$ is a Dirac delta mass.  The support and total mass of this distribution are uniquely determined by $\gamma_0$ and $\gamma_1$.  In turn, the limiting distribution
\begin{equation*}
\tfrac{1}{2k} ( d\mu (k^2) + d\mu(-k^2) )
\end{equation*}
of $\log|a|\dk$ on $\R$ is then uniquely determined by the reality condition~\eqref{eq:trans coeff 3}.  Lastly, we note that weak convergence of measures implies convergence when integrated against bounded continuous test functions by the Portmanteau theorem, and hence implies convergence in distribution.
\end{proof}

\bibliography{kdv_multisolitons}

\end{document}